\documentclass[a4paper,12pt]{amsart}
\usepackage{amssymb,graphicx}
\usepackage{color}
\usepackage{stmaryrd}
\usepackage{upgreek}
\usepackage{tensor}
\DeclareSymbolFont{script}{U}{eus}{m}{n}
\DeclareMathSymbol{\Wedge}{0}{script}{"5E}
\DeclareMathAlphabet{\mathrmsl}{OT1}{cmr}{m}{sl}
\newcommand{\R}{\mathbb R}

\newcommand{\weg}[1]{}%{{\tiny #1}}
\newcommand{\tr}{\mathrm{tr}}

\renewcommand{\c}{\mathrm{c}}
\newcommand{\nc}{\mathrm{nc}}
\newcommand{\Id}{\mathrm{Id}}
\newcommand{\Scal}{\mathrm{Scal}}
\newcommand{\Ric}{\mathrm{Ric}}

\newcommand{\gr}{\mathrm{grad}}

\newcommand{\C}{{\mathbb C}}
\newcommand{\D}{\partial}
\renewcommand{\d}{\mathrm{d}}

\theoremstyle{plain}
\newtheorem{thm}{Theorem}[section]
\newtheorem*{thm*}{Theorem}
\newtheorem{lem}[thm]{Lemma}
\newtheorem{prop}[thm]{Proposition}

\theoremstyle{definition}
\newtheorem{defn}{Definition}[section]

\newtheorem{exmp}[defn]{Example}
\theoremstyle{remark}
\newtheorem{rem}{Remark}[section]

\title[Local description of Bochner-flat (pseudo-)K\"ahler metrics]
{Local description of Bochner-flat (pseudo-)K\"ahler metrics}

\author{Alexey V. Bolsinov}
\author{Stefan Rosemann}

\numberwithin{equation}{section}

\address{Department of Mathematical Sciences, Loughborough University, Loughborough
LE11 3TU, UK}
\address{Faculty of Mechanics and Mathematics, Moscow State University, 119991
Moscow, Russia}
 \email{A.Bolsinov@lboro.ac.uk} 
\address{Institut f\"ur Differentialgeometrie, Leibniz Universit\"at Hannover, Welfengarten 1, 30167 Hannover, Germany}
\email{stefan.rosemann@math.uni-hannover.de}

\subjclass[2010]{53B20,53B30,53B35,53C35,53C25}

\begin{document}

\begin{abstract}
The Bochner tensor is the K\"ahler analogue of the conformal Weyl tensor. In this article, 
we derive local (i.e., in a neighbourhood of almost every point) normal forms for a (pseudo-)K\"ahler manifold with vanishing Bochner tensor.
The description is pined down to a new class of symmetric spaces which we describe in terms of their curvature operators.
We also give a local description of weakly Bochner-flat metrics defined by the property that the Bochner tensor has vanishing divergence.
Our results are based on the local normal forms for c-projectively equivalent metrics. As a by-product, we also describe all K\"ahler-Einstein metrics admitting a c-projectively equivalent one.
\end{abstract}

\maketitle
\setcounter{tocdepth}{1}
%\tableofcontents

\section{Introduction}

\subsection{Main results}

The Bochner tensor (see the definition~\eqref{eq:Bochner}) first appeared in~\cite[Equation (71)]{Bochner}.
Together with the trace-less Ricci tensor and the scalar curvature, it forms the three irreducible components 
into which the curvature tensor of a (pseudo-)K\"ahler manifold decomposes
under the action of the unitary group, cf.\ \cite[\S 1.2]{ApostolovI}, \cite[\S 2.63]{Besse} or \cite[\S 2.1.3]{Bryant}.
Although the definition of the Bochner tensor is formally analogous to that of the Weyl tensor, the local description of Bochner-flat metrics, i.e., 
those metrics for which the Bochner tensor vanishes identically, is much more complicated 
than the corresponding description of metrics with vanishing Weyl tensor -- recall that a metric with vanishing Weyl tensor 
locally is conformally equivalent to a flat metric (in dimension $\geq 4$).
A local classification of positive definite Bochner-flat metrics has been obtained in~\cite{ApostolovI,Bryant}.
The purpose of this article is to generalize these results to arbitrary signature: in Theorem~\ref{thm:Bochner0}, 
we give local normal forms (in a neighbourhood of a generic point) for a Bochner-flat (pseudo-)K\"ahler manifold.

These normal forms are parametrized by simple algebraic data: for the construction of a Bochner-flat (pseudo-)K\"ahler structure in complex dimension $n\geq 2$, 
choose a number $0\leq \ell\leq n$ together with a polynomial $\Theta(t)$ of degree $\leq \ell+2$ and a hermitian endomorphism $A_0$ 
on some (pseudo-)hermitian vector space $\mathbb V$ of complex dimension $k=n-\ell$ such that $\Theta(A_0)=0$. 
From this data we construct a (pseudo-)K\"ahler symmetric space $S$ of complex dimension $k$ (see Theorem~\ref{thm:symmetricspace}). 
The Bochner-flat (pseudo-)K\"ahler structure is then defined on the total space $M$ of a (local) fibration $M\rightarrow S$ over the base $S$ 
with fibers diffeomorphic to $\C^\ell$ (see Theorem~\ref{thm:Bochner0} and Remark~\ref{rem:originBochner}).

Let us be more precise about that: a (pseudo-)hermitian vector space $(\mathbb V,g_0,J_0)$ consists of a (possibly indefinite) 
scalar product $g_0$ on a real vector space $\mathbb V$ together with a $g_0$-orthogonal complex
structure $J_0:\mathbb V\rightarrow \mathbb V$. 
Let $A_0:\mathbb V\rightarrow \mathbb V$ be a hermitian (i.e., $g_0$-symmetric and $J_0$-commuting) endomorphism
and $\Theta(t)$ a polynomial with real coefficients. 
Let $\mathfrak{so}(g_0)$ denote the space of skew-symmetric endomorphisms of $(\mathbb V,g_0)$ and
consider the linear operator $\widetilde{\mathsf R}_{\Theta,A_0}:\mathfrak{so}(g_0)\rightarrow \mathfrak{so}(g_0)$ defined by
\begin{equation}\label{eq:derivofpol}
\widetilde{\mathsf R}_{\Theta,A_0}(X)=\frac{\d}{\d t}\Big|_{t=0}\Theta(A_0+tX)\quad(X\in \mathfrak{so}(g_0))
\end{equation}
(see also~\eqref{eq:polcurvopRiem0}). We will explain in \S\ref{subsec:algprelim2} that $\widetilde{\mathsf R}_{\Theta,A_0}$ 
satisfies the Bianchi identity (cf.\ also~\cite{BFom70,BolTsonev}), that is, we have
$$
0=\widetilde{\mathsf R}_{\Theta,A_0}(x\wedge y)z+\widetilde{\mathsf R}_{\Theta,A_0}(z\wedge x)y+\widetilde{\mathsf R}_{\Theta,A_0}(y\wedge z)x
\quad\mbox{for all}\quad x,y,z\in \mathbb V,
$$
where we identified $\Lambda^2\mathbb V$ and $\mathfrak{so}(g_0)$ via $x\wedge y=g_0(x,\cdot)\otimes y-g_0(y,\cdot)\otimes x$.
Hence, $\widetilde{\mathsf R}_{\Theta,A_0}$ can be identified with an algebraic curvature tensor $\widetilde{R}_{\Theta,A_0}$ 
(having the same symmetries as the curvature tensor of a Riemannian manifold).
We can project  $\widetilde{R}_{\Theta,A_0}$ onto the space of algebraic curvature tensors having the same symmetries as the K\"ahler curvature operator
(see Lemma~\ref{lem:projection} for the explicit formula of the projection operator) 
and end up with an algebraic K\"ahler curvature operator $R_{\Theta,A_0}$ (see also~\eqref{eq:defRThA} for a precise definition).
The next theorem describes the symmetric spaces which form one part of a Bochner-flat (pseudo-)K\"ahler structure:

\begin{thm}[Special symmetric spaces]\label{thm:symmetricspace}
Suppose we are given a (pseudo-)hermitian vector space $(\mathbb V,g_0,J_0)$ together with
\begin{itemize}
\item a hermitian endomorphism $A_0:\mathbb V\rightarrow \mathbb V$ and 
\item a polynomial $\Theta(t)$ such that $\Theta(A_0)=0$.
\end{itemize}
Then, the algebraic curvature tensor $R_0=R_{\Theta,A_0}$ satisfies
\begin{equation}\label{eq:condsymmetric}
0=[R_0(x,y),R_0(u,v)]-R_0\bigl(R_0(x,y)u,v\bigr)-R_0\bigl(u,R_0(x,y)v\bigr)
\end{equation}
for all $x,y,u,v\in \mathbb V$ and $[R_0(x,y),A_0]=0$ for all $x,y\in \mathbb V$. 
In particular, there exists a (pseudo-)K\"ahler symmetric space $(M,g,J)$
with curvature tensor $R$, parallel hermitian endomorphism $A$ and with base point $p$ 
such that $T_p M\cong \mathbb V$ in a certain basis and, w.r.t.\ this identification, we have
$g(p)=g_0$, $J(p)=J_0$, $A(p)=A_0$ and $R(p)=R_0$.
\end{thm}
The new information of Theorem~\ref{thm:symmetricspace} is that $R_{\Theta,A_0}$ satisfies \eqref{eq:condsymmetric}. 
The second part of the theorem is well-known (cf.\ for instance \cite[Ch.\ IV]{Helgason}): 
a symmetric space $(M,g)$ with a fixed base point $p\in M$ can be described equivalently by a (pseudo-)euclidean vector space $(\mathbb V,g_0)$ 
together with an algebraic curvature tensor $R_0$ satisfying~\eqref{eq:condsymmetric}. 
Moreover, any tensor field $T_0$ on $\mathbb V$ annihilated (in the sense of the natural $\mathfrak{gl}(\mathbb V)$-action on 
tensor spaces over $\mathbb V$) by the curvature endomorphisms $R_0(x,y)\in \mathfrak{gl}(\mathbb V)$ for all $x,y\in \mathbb V$, 
corresponds to a parallel tensor field $T$ on $(M,g)$ with $T(p)=T_0$. Note that $g_0$ and $R_0$
are annihilated by the curvature endomorphisms and that this also holds for the tensors $A_0$, $J_0$ 
from Theorem~\ref{thm:symmetricspace} w.r.t.\  the algebraic curvature tensor $R_0=R_{\Theta,A_0}$. 
Recall also that the holonomy algebra $\mathfrak h=\mathfrak{hol}_p(M,g)$ of the symmetric space $(M,g)$ at the base point $p\in M$ is given by 
$\mathfrak{h}=\mathrm{span}\{R_0(u,v):u,v\in \mathbb V\}\subseteq \mathfrak{so}(g_0)$. 

By the deRham-Wu decomposition Theorem~\cite{deRham,Wu}, in order to classify simply-connected (pseudo-)Riemannian symmetric spaces, 
it suffices to restrict to the indecomposable case, when there are no proper nondegenerate invariant subspaces for the action of $\mathfrak h$ on $(\mathbb V,g_0)$. 
Irreducible (pseudo-)Riemannian symmetric spaces are completely classified. The situation is more complicated in the indecomposable but not irreducible case. 
A general classification of pseudo-Riemannian symmetric spaces is still an open problem, see~\cite{KathOlbrich} for details. 
We plan to investigate the structure of the special (pseudo-)K\"ahler symmetric spaces from Theorem~\ref{thm:symmetricspace} in more detail in forthcoming articles. 
At the present time we do not know whether they provide new examples or whether they are all contained in existing classification results. 
We also remark that by a completely analogous procedure -- namely by ignoring the complex structure and 
simply considering the Riemannian curvature operator $\widetilde{\mathsf R}_{\Theta,A_0}$ from \eqref{eq:derivofpol} -- we would end up with a (pseudo-)Riemannian (non-K\"ahler) symmetric space. 
Our main focus in this article however lies on the description of Bochner-flat metrics which is the reason why we considered Theorem~\ref{thm:symmetricspace} only in the K\"ahler setting.

\begin{rem}[Producing explicit examples]\label{rem:symmetricspace}
Locally, the construction of the (pseudo-)Riemannian symmetric space $(M,g,T)$ via the data $(\mathbb V, g_0,T_0,R_0)$ can be obtained as follows
(although, the formulas may be hard to evaluate): 
for each $u\in \mathbb V$, let $R_u:\mathbb V\rightarrow \mathbb V$ be the endomorphism $R_u(v)=R_0(u,v)u$ and let $\varphi_u:\mathbb V\rightarrow \mathbb V$
be the isomorphism defined by 
\begin{equation}\label{eq:varphiu}
\varphi_u=\sum_{j=0}^\infty\frac{R_u^j}{(2j+1)!}.
\end{equation}
Then, one obtains a symmetric space $(M,g)$ and a parallel tensor field $T$ by setting
\begin{equation}\label{eq:symmetricu}
M=\mathbb V,\quad g(u)=\varphi_u^{-1}\cdot g_0=g_0(\varphi_u\cdot,\varphi_u\cdot) 
\quad\mbox{and}\quad
T(u)=\varphi_u^{-1}\cdot T_0,
\end{equation}
where ``$\,\cdot\,$'' denotes the standard action of $\mathrm{Gl}(\mathbb V)$ on the tensor spaces over $\mathbb V$.
Since $R_u|_{u=0}=0$ and therefore $\varphi_u|_{u=0}=\Id$, we have $g(0)=g_0$, $T(0)=T_0$ and $R(0)=R_0$ 
(where the curvature is given by $R(u)=\varphi_u^{-1}\cdot R_0$). 
\end{rem}

\begin{rem}[Freedom in the choice of the initial data]\label{rem:initialdata}
According to the deRham-Wu decomposition theorem~\cite{deRham,Wu},
the symmetric (pseudo-)K\"ahler manifold $(M,g,J)$ constructed via~Theorem \ref{thm:symmetricspace} 
decomposes (locally) into a direct product according to the generalized eigenspaces of the parallel hermitian endomorphism $A$.
Of course, this decomposition already holds at the purely algebraic level, cf.\ Proposition~\ref{prop:curvminpol}(a).
If we want to use Theorem~\ref{thm:symmetricspace} for the construction of indecomposable examples, 
we can therefore suppose that $A_0$ has only one real eigenvalue or a pair of complex conjugate eigenvalues. 
Moreover, in case of one real eigenvalue $\lambda$, we can suppose that $A_0$
is nilpotent by replacing it with $\tilde A_0=A_0-\lambda\Id$. 
Replacing $\Theta(t)$ at the same time with $\tilde{\Theta}(t)=\Theta(t+\lambda)$ yields
$\tilde\Theta(\tilde A_0)=0$ and $R_{\Theta,A_0}=R_{\tilde\Theta,\tilde A_0}$ (cf.\ Remark~\ref{rem:changepol}) 
so that the symmetric space we obtain via Theorem~\ref{thm:symmetricspace} remains unchanged.
\end{rem}

\begin{rem}[The positive definite case]\label{rem:posdefsymmetric}
For $c\in \R$ we have that $R_{\Theta,c\cdot\Id}$ is an algebraic curvature tensor of constant holomorphic sectional curvature (proportional to) $\Theta'(c)$ 
and of scalar curvature $-\mathrm{dim}_\C(\mathbb V)(\mathrm{dim}_\C(\mathbb V)+1)\Theta'(c)$ 
(cf.\ \eqref{eq:ScalRAB} and \eqref{eq:defRThA}). It follows from Remark~\ref{rem:initialdata} that for 
$g_0$ positive definite (hence, $A_0$ diagonalizable) the only examples we obtain via Theorem~\ref{thm:symmetricspace} 
are direct products of spaces of constant holomorphic sectional curvature.  
In this view, Theorem~\ref{thm:symmetricspace} shows that in the case of indefinite metrics there exist
nontrivial generalisations of such spaces related to operators which decompose into nontrivial Jordan blocks.
\end{rem}
By Remark~\ref{rem:posdefsymmetric}, nontrivial examples for the symmetric spaces from Theorem~\ref{thm:symmetricspace}
only appear in dimension $\mathrm{dim}_\C(\mathbb V)\geq 2$. The simplest example is the following:

\begin{exmp}\label{ex:symmetric}
Use, as the data $(\mathbb V,g_0,J_0,A_0,\Theta)$ in Theorem~\ref{thm:symmetricspace}, the vector space 
$\mathbb V=\R^4$ with (pseudo-)hermitian structure $(g_0,J_0)$ and hermitian endomorphism $A_0$ given by
$$
g_0=\left(\begin{array}{cccc}
0&1&0&0\\
1&0&0&0\\
0&0&0&1\\
0&0&1&0
\end{array}\right),\quad
J_0=\left(\begin{array}{cccc}
0&0&-1&0\\
0&0&0&-1\\
1&0&0&0\\
0&1&0&0
\end{array}\right),
\quad
A_0=\left(\begin{array}{cccc}
0&1&0&0\\
0&0&0&0\\
0&0&0&1\\
0&0&0&0
\end{array}\right)
$$
together with the polynomial $\Theta(t)=t^3$. Write an arbitrary point $u\in \R^4$ as $u=\sum_{i=1}^4 u_i e_i$, where $e_1,e_2,e_3,e_4$
is the standard basis of $\R^4$. From the algebraic curvature tensor $R_0=R_{\Theta,A_0}$ (computed by using~\eqref{eq:defRThA}), 
the endomorphisms $R_u=R_0(u,\cdot)u:\R^4\rightarrow \R^4$ from Remark~\ref{rem:symmetricspace} 
can be computed straight-forwardly and they all satisfy $R_u^2=0$. 
Formulas~\eqref{eq:varphiu} and~\eqref{eq:symmetricu} then give the locally symmetric space $(\R^4,g_\c,J_\c)$
with parallel hermitian endomorphism $A_\c$, where
\begin{equation}\label{eq:ex_symmetric}
g_\c=\left(\begin{array}{cccc}
0&1&0&0\\
1&\tfrac{u_4^2}{3}&0&-\tfrac{u_2u_4}{3}\\
0&0&0&1\\
0&-\tfrac{u_2u_4}{3}&1&\tfrac{u_2^2}{3}
\end{array}\right)
,\quad 
J_\c=\left(\begin{array}{cccc}
0&\tfrac{u_2u_4}{3}&-1&\tfrac{1}{6}(u_4^2-u_2^2)\\
0&0&0&-1\\
1&\tfrac{1}{6}(u_4^2-u_2^2)&0&-\tfrac{u_2 u_4}{3}\\
0&1&0&0
\end{array}\right)
,\quad 
A_\c=A_0.
\end{equation}
Note that nilpotency of all $R_u$'s implies that the symmetric space is Ricci-flat (as can be also seen directly from $R_{\Theta,A_0}$).
Of course also the choice $\Theta(t)=t^2$ satisfies $\Theta(A_0)=0$ and therefore produces a symmetric space by Theorem~\ref{thm:symmetricspace}.
However, in this case the $R_u$'s are not nilpotent in general and~\eqref{eq:symmetricu} produces more complicated formulas.
Note that the K\"ahler form $\omega_\c=g_\c(J_\c\cdot,\cdot)$ of the (pseudo-)K\"ahler structure $(g_\c,J_\c)$ 
from~\eqref{eq:ex_symmetric} is given by
\begin{equation}\label{eq:ex_symmetric2}
\omega_\c=\d u_1\wedge \d u_4+\d u_2\wedge \d u_3+\tfrac{1}{6}(u_2^2+u_4^2)\d u_2\wedge \d u_4
\end{equation}
and that it can be written as $\omega_\c=-\d\alpha$ for the 1-form
\begin{equation}\label{eq:alpha}
\alpha=\tfrac{1}{6}u_2^2 u_4\d u_2-u_2 \d u_3-(u_1+\tfrac{1}{6}u_2 u_4^2)\d u_4.
\end{equation}
The article \cite{CalFino} contains a classification of all invariant (pseudo-)K\"ahler structures 
on homogeneous four-manifolds with nontrivial isotropy subalgebra. In the symmetric setting, the isotropy subalgebra just coincides with the holonomy algebra $\mathfrak h$.
We can compute straight-forwardly that for the present example it takes the form $\mathfrak{h}=\mathrm{span}\{J_0A_0\}$. Since $\mathfrak{h}$ is $1$-dimensional and
the square of its generator $J_0A_0$ is zero, the above example is described by the case ``$1.3^1.31$ and $1.3^1.32$'' of \cite[Theorem 3.1]{CalFino}.
We also see from the matrices of $g_0$ and $J_0A_0$ that $\mathrm{span}\{e_1,e_3\}$ and its subspaces are 
the only subspaces of $\R^4$ preserved by the holonomy algebra $\mathfrak h$ and that $\mathrm{span}\{e_1,e_3\}$ is isotropic. 
Thus, the symmetric space $(g_\c,J_\c,\omega_\c)$ from \eqref{eq:ex_symmetric} and \eqref{eq:ex_symmetric2} is not irreducible 
but still it is indecomposable.
\end{exmp}
The symmetric spaces produced by Theorem~\ref{thm:symmetricspace} from the data $(g_0,J_0,A_0,\Theta)$ will be denoted by 
$\mathcal{S}_{g_0,J_0,A_0,\Theta}$. More generally, we denote by $\mathcal{S}_{g_0,J_0,A_0,\Theta}$ also open subsets 
of these spaces. Let $\mathrm{Spec}(A_0)$ denote the subset of $\C$ consisting of the eigenvalues of 
the endomorphism $A_0:\mathbb V\rightarrow \mathbb V$. By $\mathrm{Roots}(\Theta)\subseteq \C$, 
we denote the set of roots of the polynomial $\Theta(t)$. 
The local description of Bochner-flat (pseudo-)K\"ahler structures~$(g,\omega)$ is as follows:

\begin{thm}[Bochner-flat (pseudo-)K\"ahler structures]\label{thm:Bochner0}
Let $n\geq 2$, $0\leq \ell\leq n$ and consider open subsets $U,V\subseteq \R^{\ell}$. Let $t_1,\dots,t_\ell$ be real coordinates on $V$
and let $\rho_1,\dots,\rho_\ell$ be coordinates on $U$. We allow $\rho_i$ to be complex in which case it
occurs together with its complex conjugate $\rho_j=\bar{\rho}_i$ for some $j\neq i$. 
Consider the following data:
\begin{itemize}
\item A polynomial $\Theta(t)=C_2t^{\ell+2}+C_1t^{\ell+1}+\dots$ of degree $\leq \ell+2$.
\item A hermitian endomorphism $A_0$ on some $2k=2n-2\ell$-dimensional (pseudo-)hermitian vector space $(\mathbb V,g_0,J_0)$ such that
$\Theta(A_0)=0.$
\end{itemize}
Given this, we define the following:
\begin{itemize}
\item $p(t)=\prod_{i=1}^\ell(t-\rho_i)$ and $\Delta_i=p'(\rho_i)=\prod_{j\neq i}(\rho_i-\rho_j)$ (where $p'(t)=\tfrac{\d}{\d t}p(t)$). 
\item Functions $\mu_i$, $i=1,\dots,\ell$, and $\mu_i(\hat{\rho}_j)$, $i=1,\dots,\ell-1$, $j=1,\dots,\ell$, on $U$ which are
the $i$th elementary symmetric polynomials in the variables $\rho_1,\dots,\rho_\ell$ resp.\ 
$\rho_1,\dots,\hat{\rho}_j\dots,\rho_\ell$ ($\rho_j$ omitted).
\item The locally symmetric space $(S,g_\c,\omega_\c,A_\c)=\mathcal{S}_{g_0,J_0,A_0,\Theta}$.
\item The 1-forms $\theta_i=\d t_i+\alpha_i$ on $U\times V\times S$ ($i=1,\dots,\ell$), 
where $\alpha_i$ are 1-forms on $S$ such that 
\begin{equation}\label{eq:diffalpha}
\d \alpha_i=(-1)^i\omega_\c(A_\c^{\ell-i}\cdot,\cdot).
\end{equation}
(Note that for each $i$, the right-hand side is a closed 2-form on $S$.)
\end{itemize}
Then, we obtain a Bochner-flat (pseudo-)K\"ahler structure $(g,\omega)$ on the open subset 
$$
M^0=\{(\vec{\rho},\vec{t},x)\in U\times V\times S:\rho_i\neq \rho_j\,\,\forall i\neq j\,\mbox{ and }\,\rho_i\notin\mathrm{Spec}(A_0) \cup\mathrm{Roots}(\Theta)\,\,\forall i\}
$$ 
of $U\times V\times S$ which is given by the formulas
\begin{equation}\label{eq:gwBochner}
\begin{array}{c}
\displaystyle g=\sum_{i=1}^\ell \frac{\Delta_i}{\Theta(\rho_i)} \d \rho_i^2
+\sum_{i,j=1}^\ell\left[\sum_{s=1}^\ell\frac{\mu_{i-1}(\hat{\rho}_s)\mu_{j-1}(\hat{\rho}_s)}{\Delta_s}\Theta(\rho_s)\right]\theta_i\theta_j
+ g_\c(p(A_\c) \cdot , \cdot),
\vspace{1mm}\\
\displaystyle 
\omega=\sum_{i=1}^\ell \d \mu_i \wedge \theta_i + \omega_\c(p(A_c)\cdot , \cdot).
\end{array}
\end{equation}
Conversely, every $2n$-dimensional Bochner-flat (pseudo-)K\"ahler manifold $(M,g,\omega)$ 
takes locally, in a neighbourhood of a generic point, the form~\eqref{eq:gwBochner}.

The Bochner-flat (pseudo-)K\"ahler structure $(g,\omega)$ from~\eqref{eq:gwBochner}  
\begin{itemize}
\item has constant holomorphic sectional curvature with scalar curvature $\Scal=-n(n+1)C_1$ if and only if 
$\Theta(t)$ has degree $\leq \ell+1$, that is, $C_2=0$,
\item is flat if and only if $\Theta(t)$ has degree $\leq \ell$, that is, $C_2=C_1=0$,
\item is locally symmetric if and only if $\ell=0$ or $C_2=0$. 
\end{itemize}
\end{thm}

\begin{rem}[Geometric interpretation of the formulas~\eqref{eq:gwBochner}]\label{rem:originBochner}
The coordinate vector fields $\D/\D t_i$ are mutually commuting hamiltonian Killing vector fields for $(g,\omega)$ from~\eqref{eq:gwBochner}:
the coefficients of $g$ do not depend on the variables $t_i$, hence, each vector field $\D/\D t_i$ is Killing. Moreover $i_{\D/\D t_i}\omega=-\d\mu_i$
and therefore $\D/\D t_i=J\gr\,\mu_i$ is a hamiltonian vector field with hamiltonian function $\mu_i$. Since each $\D/\D t_i$ preserves $g$ and $\omega=g(J\cdot,\cdot)$,
it also preserves the complex structure $J$, i.e., each $\D/\D t_i$ is a holomorphic vector field. Then, also the vector fields $J(\D/\D t_i)$ are holomorphic. 
The symplectic reduction w.r.t.\ the (local) hamiltonian $\R^\ell$-action generated by $\D/\D t_1,\dots,\D/\D t_\ell$ 
with moment map $\vec\mu=(\mu_1,\dots,\mu_\ell):U\rightarrow \R^\ell$ yields the manifold $S$ together with the (pseudo-)K\"ahler structure $(g_\c(p(A_\c) \cdot , \cdot),J_\c)$
parametrized by the level sets of $\vec\mu$. The complex manifold $(S,J_\c)$ is the quotient of $(M,J)$ w.r.t.\ the (local) $\C^\ell$-action
generated by the holomorphic vector fields $\D/\D t_i,J(\D/\D t_i)$, $i=1,\dots,\ell$.

The $\R^\ell$-valued 1-form $\theta=(\theta_1,\dots,\theta_\ell):TM^0\rightarrow \R^\ell$ should be thought of as a connection 1-form 
on the (local) $\R^\ell$-bundle $U\times V\times S\supseteq M^0\rightarrow U\times S$ (projection onto $U\times S$). 
The condition~\eqref{eq:diffalpha} on the differentials $\d\theta_i=\d\alpha_i$ relates the curvature of this connection to the K\"ahler geometry of the base manifold $(S,g_\c,\omega_\c)$. 
It determines the 1-forms $\alpha_i$ only up to a change
$\alpha_i\to \alpha_i'=\alpha_i+\d f_i$, where the $f_i$'s are arbitrary functions on $S$. Such functions define
a ``fiber-preserving'' local transformation $f:M^0\rightarrow M^0$ given by $f(\vec \rho,\vec t,x)=(\vec \rho,t_1+f_1(x),\dots,t_\ell+f_\ell(x),x)$
which pulls back $(g,\omega)$ from~\eqref{eq:gwBochner} to the corresponding (pseudo-)K\"ahler structure written down w.r.t.\ $\theta_i'=\theta_i+\d f_i$. 

The functions $\rho_1,\dots,\rho_\ell$ also have a geometric interpretation: they are the nonconstant eigenvalues 
of the normalized Ricci tensor $\tilde \Ric=\Ric-\tfrac{\Scal}{2(n+1)}g$ of $(g,\omega)$ (when considered as hermitian endomorphism via metric duality).
More details on the geometric origin of the formulas~\eqref{eq:gwBochner} will be given in~\S\ref{ssec:methods}, Theorem~\ref{thm:localclass} and Remark~\ref{rem:naturalcoord}.
In \eqref{eq:gjwrhocoord} we will also give a formula for the complex structure $J$ corresponding to $(g,\omega)$ (in a more general context).
\end{rem}

\begin{rem}[The positive definite locally symmetric case]
Suppose $(M,g,\omega)$ is locally symmetric Bochner-flat and of nonconstant holomorphic sectional curvature, that is, 
$(M,g,\omega)$ is described by Theorem~\ref{thm:Bochner0} with $C_2\neq 0$ and $\ell=0$. In other words, $(M,g,\omega)$
takes the form of a locally symmetric space $\mathcal S_{g_0,J_0,A_0,\Theta}$ 
as described in Theorem~\ref{thm:symmetricspace} for $\Theta(t)=C_2t^2+C_1t+C_0$ a polynomial of degree $\leq 2$.
Assume further that $g$ is positive definite. By Remark~\ref{rem:posdefsymmetric}, $(M,g,\omega)=\prod_{i=1}^N(M_i,g_i,\omega_i)$, where each component $(M_i,g_i,\omega_i)$ 
has constant holomorphic sectional curvature $\Theta'(c_i)$, where $c_1,\dots,c_N$ are the different eigenvalues of the endomorphism $A_0$.
However, since $A_0$ is annihilated by the quadratic polynomial $\Theta(t)$, there are at most two such eigenvalues. Suppose that $c_1\neq c_2$
are the eigenvalues of $A_0$ such that $(M,g,\omega)=(M_1,g_1,\omega_1)\times (M_2,g_2,\omega_2)$ decomposes into two components
of constant holomorphic sectional curvatures $\Theta'(c_1)$ and $\Theta'(c_2)$ respectively. 
Since we must have $\Theta(t)=(t-c_1)(t-c_2)$ up to a nonzero factor, we see that $\Theta'(c_1)=(c_1-c_2)=-(c_2-c_1)=-\Theta'(c_2)$.
We obtain the well-known statement from~\cite{Matsumoto1} that a complex $n$-dimensional positive definite locally symmetric Bochner-flat K\"ahler manifold 
is locally a direct product $M^{n-s}_c\times M^{s}_{-c}$ for certain $0\leq s\leq n$, $c\in \R$, where $M^n_c$ 
denotes the complex $n$-dimensional space of constant holomorphic sectional curvature equal to $c$.

Note also that it was proven in \cite{Bryant} (positive definite) and \cite{CahSchwachhoefer} (arbitrary signature) that a compact Bochner-flat manifold
is always locally symmetric.
\end{rem}
According to Theorem~\ref{thm:Bochner0}, the next two examples exhaust all $4$-dimensional examples 
for a non-symmetric Bochner-flat (pseudo-)K\"ahler structure. 

\begin{exmp}[$n=\ell=2$]
Consider $\R^4$ with coordinates $\rho_1,\rho_2,t_1,t_2$, where $t_1,t_2$ are real and $\rho_1,\rho_2$ are either real or 
complex-conjugate, i.e., $\rho_2=\bar{\rho}_1$. Let $\Theta(t)=C_2t^4+C_1t^3+C_0t^2+C_{-1}t+C_{-2}$ be an arbitrary polynomial of degree $\leq 4$.
Then, $(g,\omega)$, given in the coordinates $\rho_1,\rho_2,t_1,t_2$ by the matrices 
$$
g=\left(\begin{array}{cccc}
\frac{\rho_1-\rho_2}{\Theta(\rho_1)}&0&0&0\\
0&\frac{\rho_2-\rho_1}{\Theta(\rho_2)}&0&0\\
0&0&\frac{\Theta(\rho_1)}{\rho_1-\rho_2}
+\frac{\Theta(\rho_2)}{\rho_2-\rho_1}&\frac{\rho_2 \Theta(\rho_1)}{\rho_1-\rho_2}+\frac{\rho_1\Theta(\rho_2)}{\rho_2-\rho_1}\vspace{1mm}\\
0&0&\frac{\rho_2 \Theta(\rho_1)}{\rho_1-\rho_2}+\frac{\rho_1\Theta(\rho_2)}{\rho_2-\rho_1}&\frac{\rho_2^2\Theta(\rho_1)}{\rho_1-\rho_2}+\frac{\rho_1^2\Theta(\rho_2)}{\rho_2-\rho_1}
\end{array}\right),
$$
$$
\omega=\d (\rho_1+\rho_2)\wedge \d t_1 +\d(\rho_1\rho_2)\wedge \d t_2
=\left(
\begin{array}{cccc}
0&0&1&\rho_2\\
0&0&1&\rho_1\\
-1&-1&0&0\\
-\rho_2&-\rho_2&0&0
\end{array}
\right),
$$
is a Bochner-flat (pseudo-)K\"ahler structure on the open subset of $\R^4$ where $\rho_1\neq \rho_2$ and
$\rho_1,\rho_2\notin\mathrm{Roots}(\Theta)$.
\end{exmp}
Note that in the previous example, the symmetric space $(S,g_\c,\omega_\c,A_\c)=\mathcal{S}_{g_0,J_0,A_0,\Theta}$ from Theorem~\ref{thm:symmetricspace} 
does not appear. In the next example, it does appear but it is complex $1$-dimensional and therefore still takes a trivial form
(the hermitian endomorphism $A_0$ is a real multiple of $\Id$ and we put it equal to zero, cf.\ Remark~\ref{rem:initialdata}):

\begin{exmp}[$n=2$, $\ell=1$]
Let $\Theta(t)=C_2t^{3}+C_1t^2+C_0t$ be an arbitrary polynomial of degree $\leq 3$ such that $\Theta(0)=0$ 
and let $(S,g_c,\omega_c)$ be an oriented constant curvature surface. 
Choose a 1-form $\alpha$ on $S$ such that $\d\alpha=-\omega_\c$ and let $\rho,t$ be coordinates on $\R^2$.
Then, $(g,\omega)$ given by 
$$
g=\frac{\d\rho^2}{\Theta(\rho)}+\Theta(\rho)(\d t+\alpha)^2+\rho \,g_\c,
\quad
\omega=\d\rho\wedge(\d t+\alpha)+\rho\,\omega_\c
$$
is a Bochner-flat (pseudo-)K\"ahler structure on the open subset of $\R^2\times S$ where $\rho\neq 0$ and $\rho\notin \mathrm{Roots}(\Theta)$. 
For writing it down more explicitly in terms of $4\times 4$ matrices, choose coordinates $x_1,x_2$ on $S$ 
such that $\alpha=\alpha_1\d x_1+\alpha_2\d x_2$ (for certain functions $\alpha_1,\alpha_2$ on $S$)
and consider the components $(g_\c)_{ij}=g_\c(\D_{x_i},\D_{x_j})$ and $(\omega_\c)_{12}=\omega_\c(\D_{x_1},\D_{x_2})$ 
of $g_\c$ and $\omega_\c$ respectively. 
Then, $(g,\omega)$ in the coordinates $\rho,t,x_1,x_2$ takes the form
$$
g=\left(\begin{array}{cccc}
\frac{1}{\Theta(\rho)}&0&0&0\\
0&\Theta(\rho)&\alpha_1\Theta(\rho)&\alpha_2\Theta(\rho)\\
0&\alpha_1\Theta(\rho)&\alpha_1^2\Theta(\rho)+\rho\,(g_\c)_{11}&\alpha_1\alpha_2\Theta(\rho)+\rho\,(g_\c)_{12}\\
0&\alpha_2\Theta(\rho)&\alpha_1\alpha_2\Theta(\rho)+\rho\,(g_\c)_{12}&\alpha_2^2\Theta(\rho)+\rho\,(g_\c)_{22}
\end{array}\right),
$$
$$
\omega=\left(\begin{array}{cccc}
0&1&\alpha_1&\alpha_2\\
-1&0&0&0\\
-\alpha_1&0&0&\rho(\omega_\c)_{12}\\
-\alpha_2&0&-\rho(\omega_\c)_{12}&0
\end{array}\right).
$$
\end{exmp}
The simplest example in which the symmetric spaces $\mathcal{S}_{g_0,J_0,A_0,\Theta}$ from Theorem~\ref{thm:symmetricspace} 
enter in a nontrivial way is the following:

\begin{exmp}[$n=3$, $\ell=1$]
As before let $\rho,t$ be coordinates on $\R^2$. 
We use the symmetric space $(S=\R^4,g_\c,\omega_\c,A_\c)=\mathcal{S}_{g_0,J_0,A_0,\Theta}$ from formulas~\eqref{eq:ex_symmetric} and~\eqref{eq:ex_symmetric2}
of Example~\ref{ex:symmetric} as part of the data in Theorem~\ref{thm:Bochner0} and obtain the $6$-dimensional Bochner-flat (pseudo-)K\"ahler structure
$$
g=\frac{\d \rho^2}{\rho^3} +\rho^3 (\d t+\alpha)^2+ g_\c((A_\c-\rho\Id) \cdot , \cdot),\quad
\omega=\d \rho \wedge (\d t+\alpha) + \omega_\c((A_\c-\rho\Id)\cdot , \cdot)
$$
on the open subset of $\R^2\times S$ where $\rho\neq 0$. The 1-form $\alpha$ is given by~\eqref{eq:alpha}.
\end{exmp}

\subsection{Previous results and methods}
\label{ssec:methods}

Almost all previously known results on Bochner-flat metrics (cf.\ for instance \cite{Ejiri,Kami2,Kami,Matsumoto1,Matsumoto0,Tachibana1,Tachibana0}) 
have been generalized or rediscovered in \cite{Bryant}. Among the already mentioned local description, the article
\cite{Bryant} also contains a classification of complete Bochner-flat metrics and, in particular, a description of the compact case: 
it turns out that compact Bochner-flat manifolds are locally symmetric. This statement also follows immediately from combining the local classification results from~\cite{ApostolovI} 
with the compactification conditions obtained in~\cite{ApostolovII}.
The result that compact Bochner-flat manifolds are locally symmetric appeared first in~\cite{Kami} but its proof turned out to be incomplete, cf.\ \cite{Kami2}. 
The proof in~\cite{Kami} was based on the fact that all Bochner-flat metrics can be obtained (locally) by taking the quotient of a flat (as a Cartan geometry) CR-structure 
w.r.t.\ the flow of a CR-vector field which is transversal to the Levi distribution. This observation is originally due to~\cite{Webster} (see also~\cite{David,DavidGauduchon,PanakSchwachhoefer})  
and this point of view was also taken up in~\cite{CahSchwachhoefer} from a more general perspective to prove that also compact Bochner-flat (pseudo-)K\"ahler manifolds are locally symmetric. 
Theorem~\ref{thm:Bochner0} provides a local description of Bochner-flat (pseudo-)K\"ahler manifolds on neighbourhoods of generic points which are characterized by some regularity condition, 
cf.\ Remark~\ref{rem:generic}. It would still be very interesting to obtain a description of all complete Bochner-flat (pseudo-)K\"ahler manifolds and to extend the local description to singular points.

Our approach towards Bochner-flat (pseudo-)K\"ahler metrics is different from that taken in 
\cite{Bryant}. The latter was based on the analysis of the exterior differential system describing a Bochner-flat K\"ahler structure whilst we will follow the approach taken in \cite{ApostolovI}. 
The authors of \cite{ApostolovI} derived local normal forms for K\"ahler metrics with a so-called hamiltonian 2-form
(see \cite[Equation (1)]{ApostolovI} for the defining equation) and observed in \cite[\S 1.3]{ApostolovI} (see also Lemma \ref{lem:Bochnertensor} below), 
that a K\"ahler metric is weakly Bochner-flat, i.e, its Bochner-tensor has vanishing divergence, if and only if the so-called normalized Ricci form is a hamiltonian 2-form. 
Subsequently, they found the additional conditions on the parameters of their normal forms such that the metric is weakly Bochner-flat resp.\ Bochner-flat.
Our approach in the (pseudo-)K\"ahler setting is the same, although, our language will be slightly different: instead of a $J$-invariant 2-form
$\varphi$, we consider the associated hermitian (i.e., $g$-symmetric, $J$-commuting) endomorphism $A$ 
defined by $\varphi=g(JA\cdot,\cdot)$. The 2-form $\varphi$ is hamiltonian (i.e., satisfies \cite[Equation (1)]{ApostolovI}) 
if and only if $A$ satisfies the linear first order partial differential equation
\begin{equation}\label{eq:main}
\nabla_u A=u^\flat\otimes \Lambda+\Lambda^\flat\otimes u+(Ju)^\flat\otimes J\Lambda+(J\Lambda)^\flat\otimes Ju
\end{equation}
for all $u\in TM$, where $u^\flat=g(u,\cdot)$, $\Lambda=\tfrac{1}{4}\gr\,\tr(A)$ and $\nabla$ denotes the Levi-Civita connection of $g$. 
The solutions of \eqref{eq:main} have a nice geometric interpretation which in the past motivated us and many other authors to
study them without even knowing about the equivalent (and much younger) theory of hamiltonian 2-forms, 
see \cite{FKMR,MatRos,YanoObata,Mikes,Yanobook} and references therein: nondegenerate hermitian solutions of \eqref{eq:main} are in one-one
correspondence to metrics $\hat g$ which are c-projectively equivalent to $g$. Recall that two (pseudo-)K\"ahler metrics $g,\hat g$
on a complex manifold $(M,J)$ are called \emph{c-projectively equivalent} if their $J$-planar curves coincide. These curves can be viewed as some kind
of generalized complex geodesics, see \cite[Equation (1.1)]{BMR} for a precise definition. For the explicit relation between solutions $A$ of \eqref{eq:main}
and metrics $\hat g$, c-projectively equivalent to $g$, see \cite[Equations (1.2), (1.3)]{BMR}.
Independently of this and with different motivation, the theory of hamiltonian 2-forms was developed locally and globally in \cite{ApostolovI,ApostolovII}.
In the more recent articles \cite{BMR,BMMR,CEMN,CMR} ideas of both theories have been combined.

\medskip
In \cite{BMR} we locally classified c-projectively equivalent (pseudo-)K\"ahler metrics generalizing the results of~\cite{ApostolovI}
to the case of arbitrary signature: we derived normal forms for a (pseudo-)K\"ahler manifold 
(locally in a neighbourhood of a generic point) which admits a hermitian solution of~\eqref{eq:main}. 
This result will be recalled in formulas \eqref{eq:gjwrhocoord} and \eqref{eq:Arhocoord} of Theorem~\ref{thm:localclass}. 
In Lemma~\ref{lem:Bochnertensor} we recall the fundamental observation made in \cite[\S 1.3]{ApostolovI} 
that the normalized Ricci tensor (see the defining Equation~\eqref{eq:normRic}) is a solution of~\eqref{eq:main}
if and only if the (pseudo-)K\"ahler structure is weakly Bochner-flat.
We then continue to derive the additional conditions on the parameters of the normal forms~\eqref{eq:gjwrhocoord} 
such that the (pseudo-)K\"ahler structure becomes K\"ahler-Einstein resp.\ weakly-Bochner flat. 
This is the content of Theorem~\ref{thm:WBF}. The description of weakly Bochner-flat (pseudo-)K\"ahler metrics from Theorem~\ref{thm:WBF} is reduced 
to (pseudo-)K\"ahler metrics with parallel Ricci tensor (as studied for instance in~\cite{Boubel3}).
Of course, by the DeRham decomposition theorem~\cite{deRham}, such metrics in case they are positive definite just decompose (locally) into a direct product of K\"ahler-Einstein metrics. 
In arbitrary signature however, they may be indecomposable and non-K\"ahler-Einstein, cf.\ Remark \ref{rem:nonEinstein}. 
The articles \cite{Boubel2,Boubel} contain a description of (pseudo-)Riemannian metrics with parallel endomorphisms but applying this to obtain a 
description of (pseudo-)K\"ahler metrics with parallel Ricci tensor does not seem to be straight-forward.

\medskip
The language which is necessary for the precise description of the algebraic curvature tensors of the form $R_{\Theta,A_0}$
appearing in Theorems~\ref{thm:symmetricspace} and~\ref{thm:Bochner0} will be introduced in~\S\ref{subsec:algprelim2}. 
In that section we will also prove Theorem~\ref{thm:symmetricspace} which in fact is a simple algebraic observation (cf.\ Proposition~\ref{prop:curvminpol}(c)).
The ideas for dealing with the curvature operators ${\mathsf R}_{\Theta,A_0}:\mathfrak{u}(g_0,J_0)\rightarrow \mathfrak{u}(g_0,J_0)$ 
already appeared in previous works: we will see that they are solutions of an equation on operators 
on the Lie algebra $\mathfrak{u}(g_0,J_0)$, see Equations~\eqref{eq:secop0} and~\eqref{eq:secop}. Remarkably Equation~\eqref{eq:secop0}, 
which makes sense for an arbitrary semisimple Lie algebra $\mathfrak{g}$, 
occurs in several a priori unrelated fields: it first appeared in the theory of finite-dimensional integrable systems on Lie algebras~\cite{BFom70,Man,Mish}
but has also applications to (real) projective geometry~\cite[\S3]{Fubini}, c-projective geometry~\cite[Appendix A]{BMR} and to the construction of certain 
(pseudo-)Riemannian holonomy algebras~\cite{BolTsonev}. The interplay between solutions of~\eqref{eq:secop} and symmetric spaces, as discovered in this article, 
will be explored in more detail in forthcoming articles.

\medskip
Our main achievement in extending the results of~\cite{ApostolovI,Bryant} to the (pseudo-)K\"ahler case
is that we were able to deal with the Jordan blocks of the normalized Ricci tensor. 
As stated above, this tensor gives a solution of~\eqref{eq:main} if and only if the (pseudo-)K\"ahler manifold is weakly Bochner-flat. 
We showed in~\cite{BMR} that nontrivial Jordan blocks of solutions of~\eqref{eq:main} corresponding to nonconstant eigenvalues
(when the eigenvalues are considered as functions on the manifold) cannot appear -- 
this was one of the key ingredients in~\cite{BMR} for extending the local classification results from~\cite{ApostolovI} to arbitrary signature.
On the other hand, in the Bochner-flat case, the normalized Ricci tensor in general will have Jordan blocks corresponding to constant eigenvalues
which give rise to the special class of symmetric spaces from Theorem~\ref{thm:symmetricspace}.
After having described the weakly Bochner-flat metrics among the metrics from~\eqref{eq:gjwrhocoord}
(Theorem~\ref{thm:WBF}), we continue with deriving the conditions under which these metrics are actually Bochner-flat. 
This is the content of Theorem~\ref{thm:Bochner} and Theorem~\ref{thm:Bochner0} will follow from that.

\section{Background material, additional results and precise formulations}

\subsection{Description of c-projectively equivalent (pseudo-)K\"ahler metrics}

The local description of Bochner-flat (pseudo-)K\"ahler metrics from Theorem~\ref{thm:Bochner0} follows from Theorem~\ref{thm:Bochner}.
The latter is based on the local description \cite[Example 5 and Theorem 1.6]{BMR} of $2n$-dimensional 
(pseudo-)K\"ahler manifolds $(M,g,J)$ which admit a hermitian solution $A$ of~\eqref{eq:main}. 
We will recall this description in Theorem~\ref{thm:localclass} below in slightly different notation and coordinates
compared to~\cite[Example 5]{BMR}. This reformulation can also be found in \cite[Theorem 5.21]{CEMN}. 
We start by explaining how Theorem~\ref{thm:localclass} is obtained from \cite[Example 5 and Theorem 1.6]{BMR}.
The reader not familiar with the results of~\cite{BMR} may ignore the next few lines and can directly jump to 
Theorem~\ref{thm:localclass} and the formulas~\eqref{eq:gjwrhocoord} and~\eqref{eq:Arhocoord} therein.
A geometric interpretation of these formulas is given in Remark~\ref{rem:naturalcoord}.
As in \cite{ApostolovI} we will use the nonconstant eigenvalues $\rho_1,\dots,\rho_\ell$ of $A$ as 
coordinates instead of the $z_1,\dots,z_r,x_{r+1},\dots,x_{r+q}$ that were used in \cite[Example 5]{BMR}. 
Compared to~\cite[Example 5]{BMR}, we have also changed the numbering of the $\rho_i$'s to make our notation more transparent: 
the set $\{\rho_1,\dots,\rho_\ell\}$ contains \emph{all} nonconstant 
eigenvalues of $A$, that is, this time complex conjugate eigenvalues are counted separately. If $\rho_i$ is real, we define 
$$
\Theta_i(\rho_i):=\varepsilon_i\left(\frac{\D \rho_i}{\D x_i}\right).
$$
If $\rho_i,\rho_j=\bar\rho_i$ is a pair of complex conjugate eigenvalues, we define
$$
\Theta_i(\rho_i):=-4\left(\frac{\D \rho_i}{\D z_i}\right)^2\mbox{ and }\Theta_j(\rho_j):=\overline{\Theta_i(\rho_i)}.
$$
Note that $\Theta_i(z)$ is a holomorphic function of $z$, i.e., $\tfrac{\D}{\D \bar z}\Theta_i(z)=0$. 
The local classification \cite[Example 5 and Theorem 1.6]{BMR} now takes the following equivalent form:

\begin{thm}[Description of c-projectively equivalent metrics]\label{thm:localclass}
Let $n\geq 2$, $0\leq \ell\leq n$ and consider open subsets $U,V\subseteq \R^{\ell}$. Let $t_1,\dots,t_\ell$ be real coordinates on $V$
and let $\rho_1,\dots,\rho_\ell$ be coordinates on $U$. We allow $\rho_i$ to be complex in which case it
occurs together with its complex conjugate $\rho_j=\bar{\rho}_i$ for some $j\neq i$. 
Consider the following data:
\begin{itemize}
\item For each $i=1,\dots,\ell$ a function $\Theta_i$ of one real (resp.\ complex) variable which is defined and nowhere zero
on the image of $\rho_i$ and is a smooth (resp.\ holomorphic) function of $\rho_i$ if $\rho_i$ is real (resp.\ complex). 
Moreover, for a complex conjugate pair $\rho_j=\bar{\rho_i}$ we require $\Theta_j(\rho_j)=\overline{\Theta_i(\rho_i)}$. 
\item A $2k=2n-2\ell$-dimensional (pseudo-)K\"ahler manifold $(S,g_\c,J_\c,\omega_\c)$ 
with parallel hermitian endomorphism $A_\c$. We write $J_\c=\sum_{r,s=1}^{2k}(J_\c)^r_s \D_{x_r}\otimes \d x_s$ and
$A_\c=\sum_{r,s=1}^{2k}(A_\c)^r_s \D_{x_r}\otimes \d x_s$ in (arbitrary) local coordinates $x_1,\dots,x_{2k}$ on $S$.
\end{itemize}
Given this, define 
\begin{itemize}
\item $p_\nc(t)=\prod_{i=1}^\ell(t-\rho_i)$ and $\Delta_i=p_\nc'(\rho_i)=\prod_{j\neq i}(\rho_i-\rho_j)$ (where $p_\nc'(t)=\tfrac{\d}{\d t}p_\nc(t)$).
\item functions $\mu_i$, $i=1,\dots,\ell$, and $\mu_i(\hat{\rho}_j)$, $i=1,\dots,\ell-1$, $j=1,\dots,\ell$, which are the $i$th elementary symmetric 
polynomials in the variables $\rho_1,\dots,\rho_\ell$ resp.\ $\rho_1,\dots,\hat{\rho}_j\dots,\rho_\ell$ ($\rho_j$ omitted),
\item 1-forms $\theta_i=\d t_i+\alpha_i$ on $U\times V\times S$, $i=1,\dots,\ell$, 
where $\alpha_i=\sum_{r=1}^{2k}\alpha_{ir}\d x_r$ are 1-forms on $S$ such that 
$$
\d \alpha_i=(-1)^i\omega_\c(A_\c^{\ell-i}\cdot,\cdot).
$$
(Note that for each $i$, the right-hand side is a closed 2-form on $S$.)
\end{itemize}
Then, we obtain a (pseudo-)K\"ahler structure $(g,J,\omega)$ on the open subset 
$$
\{(\vec{\rho},\vec{t},x)\in U\times V\times S:\rho_i\neq \rho_j\,\,\forall i\neq j\,\mbox{ and }\,\rho_i\notin\mathrm{Spec}(A_\c)\,\,\forall i\}
$$ 
of $U\times V\times S$ which is given by the formulas
\begin{equation}\label{eq:gjwrhocoord}
\begin{array}{c}
\displaystyle g=\sum_{i=1}^\ell \frac{\Delta_i}{\Theta_i(\rho_i)} \d \rho_i^2
+\sum_{i,j=1}^\ell\left[\sum_{s=1}^\ell\frac{\mu_{i-1}(\hat{\rho}_s)\mu_{j-1}(\hat{\rho}_s)}{\Delta_s}\Theta_s(\rho_s)\right]\theta_i\theta_j
+ g_\c(p_\nc(A_\c) \cdot , \cdot),\vspace{1mm}\\
\displaystyle \d \rho_i\circ J=-\frac{\Theta_i(\rho_i)}{\Delta_i}\sum_{j=1}^\ell\mu_{j-1}(\hat{\rho_i})\theta_j,\quad
\theta_i\circ J=(-1)^{i-1}\sum_{j=1}^\ell \frac{\rho_{j}^{\ell-i}}{\Theta_j(\rho_j)}\d \rho_j,\quad \d x_i\circ J=\d x_i\circ J_\c,
\vspace{1mm}\\
\displaystyle \omega=\sum_{i=1}^\ell \d \mu_i \wedge \theta_i + \omega_\c(p_\nc(A_c)\cdot , \cdot).
\end{array}
\end{equation}
Moreover, the endomorphism $A$ given by
\begin{equation}\label{eq:Arhocoord}
\displaystyle A=\sum_{i,j=1}^\ell\left(\mu_i\delta_{1j}-\delta_{i(j-1)}\right)\theta_i\otimes \frac{\D }{\D t_j}
+\sum_{i=1}^\ell \rho_i \d \rho_i\otimes \frac{\D}{\D \rho_i}\vspace{1mm}\\
\displaystyle
+\sum_{r,s=1}^{2k}    (A_\c)^r_s  \,  \d x_r \otimes \left( \frac{\D}{\D x_s}  - \sum_{i=1}^\ell\alpha_{is} \frac{\D}{\D t_i}   \right) 
\end{equation}
is a hermitian solution of \eqref{eq:main} with nonconstant eigenvalues $\rho_1,\dots,\rho_\ell$ and constant eigenvalues $\mathrm{Spec}(A_\c)$.

Conversely, every $2n$-dimensional (pseudo-)K\"ahler manifold $(M,g,J,\omega)$ with hermitian solution $A$ of \eqref{eq:main}
takes locally the form \eqref{eq:gjwrhocoord} and \eqref{eq:Arhocoord} in a neighbourhood of a generic point.
\end{thm}

\begin{rem}[Generic points]\label{rem:generic}
Let us be more precise about the notion of a ``generic point'' which occurred at the end of Theorem~\ref{thm:localclass}:
let $(M,g,J)$ be a connected (pseudo-)K\"ahler manifold and let $A$ be a hermitian solution  
of~\eqref{eq:main}. We call a point $p\in M$ \emph{generic}, if in a neighbourhood of this point the number of different eigenvalues of $A$ is constant 
(which implies that the eigenvalues are smooth functions on some neighbourhood of $p$), and 
for each eigenvalue $\rho$ either $\d\rho\ne 0$ at $p$, or $\rho$ is constant on a neighbourhood of $p$.   
By definition, the  set $M^0$ of generic points  is open and dense in $M$.
By \cite[Lemma 2.2(4)]{BMR}, the number of nonconstant eigenvalues of $A$ is the same at each generic point. 
\end{rem}

\begin{rem}[Decomposing the ``constant block'']
By the deRham-Wu decomposition theorem~\cite{deRham,Wu}, we can write
\begin{equation}\label{eq:deRhamdecomp}
(S,g_c,\omega_c)=(S_1,g_1,\omega_1)\times\dots\times(S_N,g_N,\omega_N)
%\prod_{\gamma=1}^N(S_\gamma,g_\gamma,\omega_\gamma)
,
\quad A_\c=\sum_{\gamma=1}^N A_\gamma,
\end{equation}
where $A_\gamma$ is a $g_\gamma$-parallel hermitian endomorphism on the (pseudo-)K\"ahler manifold $(S_\gamma,g_\gamma,\omega_\gamma)$ 
whose spectrum consists either of one single real eigenvalue $c_\gamma$ (in which case the complex dimension of $S_\gamma$ equals
the multiplicity of $c_\gamma$ in the characteristic polynomial $p_\c(t)=\det_\C(t\cdot\Id-A_\c)$ of $A_\c$ considered as complex endomorphism) 
or of two complex conjugate eigenvalues $c_\gamma,\bar{c}_\gamma$ (in which case the complex dimension of $S_\gamma$ equals two times the multiplicity of $c_\gamma$ in $p_\c(t)$).
\end{rem}

\begin{rem}[The coordinates are natural]\label{rem:naturalcoord}
We see that the vector fields $\D/\D t_i$ from Theorem~\ref{thm:localclass} are mutually commuting holomorphic Killing vector fields for $(g,J,\omega)$
from \eqref{eq:gjwrhocoord}: they preserve the metric and the complex structure and, hence, also the K\"ahler form.
The existence of these Killing vector fields has been the key tool for obtaining the local classification of hamiltonian 2-forms 
resp.\ c-projectively equivalent metrics, cf.\ \cite[Proposition 3]{ApostolovI} resp.\ \cite[Lemma 2.1]{BMR},
and also for the classification of Bochner-flat K\"ahler metrics \cite[Theorem 3.11]{Bryant}.
Note that each $\D/\D t_i$ is a globally defined hamiltonian vector field with hamiltonian function $\mu_i$
(where $\mu_i$ is the $i$th elementary symmetric polynomial in the nonconstant eigenvalues $\rho_1,\dots,\rho_\ell$ of $A$, hence, a globally defined function on $M$): 
by \eqref{eq:gjwrhocoord} we have $i_{\D/\D t_i}\omega=-\d\mu_i$. 
Also we see from~\eqref{eq:Arhocoord} that the (generalized) $\rho_i$-eigenspace of $A$ is spanned by 
$\D/\D \rho_i,J(\D/\D \rho_i)$. In particular, $A$ has no Jordan blocks corresponding to nonconstant eigenvalues 
(as predicted by the general theory, cf.\ \cite[Lemma 2.2]{BMR}). From~\eqref{eq:gjwrhocoord} we also see that $\D/\D \rho_i=\frac{\Delta_i}{\Theta(\rho_i)}\gr\,\rho_i$.
Thus, the $\rho_i$-eigenspace of a solution $A$ of~\eqref{eq:main} is spanned by the vector fields $\gr\,\rho_i$ and $J\gr\,\rho_i$.
The eigenvalues $\mathrm{Spec}(A_\c)$ of $A_\c$ are the constant eigenvalues of $A$. Indeed, if $u\in TS$ is contained in the generalized eigenspace of $A_\c$ 
corresponding to some $\xi\in \mathrm{Spec}(A_\c)$, it ``lifts'' to an element $u-\sum_{i=1}^\ell \alpha_i(u)\D/\D t_i$ contained in the generalized $\xi$-eigenspace of $A$. 
In particular, the characteristic polynomial of $A$ considered as complex endomorphism splits into the product of $p_\nc(t)=\prod_{i=1}^\ell(t-\rho_i)$ 
with the characteristic polynomial $p_\c(t)=\det_\C(t\cdot \Id-A_\c)$ of $A_\c$ considered as complex endomorphism:
$$
\mathrm{det}_\C(t\cdot \Id-A)=p_\nc(t)p_\c(t).
$$
We also see that the parameters $\Theta_i$ appearing in~\eqref{eq:gjwrhocoord} satisfy $\Theta_i(\rho_i)=\Delta_ig(\gr\,\rho_i,\gr\,\rho_i)$, 
hence, on the image of $\rho_i$, $\Theta_i$ is determined by the eigenvalue structure of the solution $A$ of~\eqref{eq:main} given by~\eqref{eq:Arhocoord}.
Finally, the (pseudo-)K\"ahler manifold $(S,g_\c(p_\nc(A_\c)\cdot,\cdot),J_\c)$ is the (local) K\"ahler quotient of $(M,g,J)$ 
w.r.t.\ the hamiltonian action of (the (pseudo-)group generated by) the hamiltonian Killing vector fields $\D/\D t_1,\dots,\D/\D t_\ell$ 
with moment map $\vec\mu=(\mu_1,\dots,\mu_\ell):M\rightarrow \R^\ell$. As a complex manifold, $(S,J_\c)$ is the (local) quotient of $(M,J)$ 
w.r.t.\ the holomorphic action of (the (pseudo-)group generated by) the holomorphic vector fields $\D/\D t_i,J(\D/\D t_i)$, $i=1,\dots,\ell$,
see \ \cite[\S 3.1]{ApostolovI} and \cite[\S 3.1]{BMR}. 
\end{rem}

\begin{rem}[Relation to the results of \cite{ApostolovI}]
The main result of \cite{ApostolovI} describes the positive definite K\"ahler structures admitting a hamiltonian 2-form.
Such a 2-form is given by $\phi=g(JA\cdot,\cdot)$, where $A$ is a hermitian solution of \eqref{eq:main}.
Contracting $A$ from \eqref{eq:Arhocoord} with $g$ from \eqref{eq:gjwrhocoord} shows 
%\begin{equation}\label{eq:Arhocoord0}
$$
g(A\cdot,\cdot)=\sum_{i=1}^\ell\rho_i\frac{\Delta_i}{\Theta_i(\rho_i)}\d\rho_i^2
+\sum_{i,j=1}^\ell\left(\sum_{s=1}^\ell\rho_s\frac{\mu_{i-1}(\hat{\rho}_s)\mu_{j-1}(\hat{\rho}_s)}{\Delta_s}\Theta_s(\rho_s)\right)\theta_i\theta_j
+g_\c(p_\nc(A_\c)A_\c\cdot,\cdot).
$$
%\end{equation}
Contracting this with $J$ from \eqref{eq:gjwrhocoord}, we get 
\begin{equation}\label{eq:hamrhocoord}
\phi=\sum_{i,j=1}^\ell\rho_i\mu_{j-1}(\hat{\rho}_i)\d\rho_i\wedge \theta_j+\omega_\c(p_\nc(A_\c)A_\c\cdot,\cdot)
=-\sum_{i=1}^\ell\rho_i\frac{\Delta_i}{\Theta_i(\rho_i)}\d\rho_i\wedge \d\rho_i\circ J+\omega_\c(p_\nc(A_\c)A_\c\cdot,\cdot).
\end{equation}
The formulas \eqref{eq:gjwrhocoord} and \eqref{eq:hamrhocoord} look almost identical to the ones from 
\cite[Theorem 1]{ApostolovI}: the main difference (besides notation) 
is that in our case some of the nonconstant eigenvalues $\rho_i$ are complex with corresponding holomorphic 
functions $\Theta_i$ and that the hermitian parallel endomorphism $A_c$ on $(S,g_\c,\omega_\c)$ is in general 
not diagonalizable, that is, $A_c$ may have nontrivial Jordan blocks.
Note that the functions $F_i$ appearing in the explicit formulas from \cite[Theorem 1]{ApostolovI} 
are given by $F_i(t)=\Theta_i(t)p_\c(t)$, where $p_\c(t)=\det_\C(t\cdot \Id-A_\c)$ is the characteristic polynomial of $A_\c$ considered as complex endomorphism.
\end{rem}

\subsection{The Bochner tensor and c-projectively equivalent (pseudo-)K\"ahler metrics}
\label{subsec:Bochnertensor}

For $(0,2)$-tensors $A,B\in \bigotimes^2\mathbb{V}^*$ on some vector space $\mathbb V$, the Kulkarni-Nomizu product 
$A\owedge B\in \bigotimes^4\mathbb{V}^*$ is the $(0,4)$-tensor defined by
$$
(A\owedge B)(x,y,u,v)=A(x,u)B(y,v)-A(x,v)B(y,u)+B(x,u)A(y,v)-B(x,v)A(y,u)
$$
for $x,y,u,v\in \mathbb{V}$. We also define the symmetric product $A\odot B=A\otimes B+B\otimes A$ of $(0,2)$-tensors $A,B$.

Now consider a (pseudo-)K\"ahler structure $(g,J)$ of real dimension $2n$ with K\"ahler form $\omega=g(J\cdot,\cdot)$. 
We will view the curvature tensor $R$ of $g$ either as a $(1,3)$-tensor via $R(x,y)z=(\nabla_{x}\nabla_{y}-\nabla_{y}\nabla_{x}-\nabla_{[x,y]})z$ 
or as a $(0,4)$-tensor given by $R(x,y,u,v)=g(x,R(u,v)y)$. The Ricci tensor $\Ric$, the Ricci form $r$ and the scalar curvature $\Scal$
are given by $\Ric(x,y)=\tr(z\longmapsto R(z,x)y)$, $r(x,y)=\mathrm{Ric}(Jx,y)$ and $\Scal=\tr(\Ric)$ respectively. In the formula for $\Scal$
we adopt the convention to identify a hermitian (i.e., symmetric, $J$-invariant) $(0,2)$-tensor $A$ with the corresponding hermitian (i.e., $g$-symmetric, $J$-commuting) 
endomorphism (both denoted by the same symbol) via $A(x,y)=g(x,Ay)$. The Bochner tensor $\mathrm{Boch}$ is defined by

\begin{equation}\label{eq:Bochner}
\begin{array}{c}
\displaystyle R=\mathrm{Boch}+\frac{1}{2(n+2)}[g\owedge \mathrm{Ric}+\omega\owedge r
+2\omega \odot r]
-\frac{\mathrm{Scal}}{8(n+1)(n+2)}[g\owedge g+\omega\owedge \omega+2\omega\odot \omega].
\end{array}
\end{equation}
Note that in order to obtain the decomposition of $R$ into irreducible components w.r.t.\ the action of the unitary group $\mathrm{U}(g,J)$,
we have to replace $\Ric$ in~\eqref{eq:Bochner} by its trace-free part $\Ric_0=\Ric-\tfrac{\Scal}{2n}g$.
The divergence of a $(0,s+1)$-tensor field $T$ is defined to be the $(0,s)$-tensor field $\mathrm{Div}(T)$
which, in abstract index notation, is given by $\mathrm{Div}(T)_{i_1\dots i_s}=g^{ij}\nabla_iT_{ji_1\dots i_s}$.
A (pseudo-)K\"ahler structure is called \emph{weakly Bochner-flat} if $\mathrm{Div}(\mathrm{Boch})=0$. It is called \emph{Bochner-flat}
if $\mathrm{Boch}=0$. 

Let us define the \emph{normalized Ricci tensor} $\tilde{\mathrm{Ric}}$ by
\begin{equation}\label{eq:normRic}
\tilde{\mathrm{Ric}}=\mathrm{Ric}-\frac{\mathrm{Scal}}{2(n+1)}g.
\end{equation}
The following statement was proven in \cite[\S 1.3]{ApostolovI} for positive definite $g$. 
Its proof carries over without any change to the case of arbitrary signature (the computation of the divergence of~\eqref{eq:Bochner}
only involves contractions and some standard identities for the K\"ahler curvature tensor which hold in any signature).
We refer also to \cite[\S 2.3 and Equation (2.14)]{Bryant} for the corresponding statement in the Bochner-flat case.

\begin{lem}\label{lem:Bochnertensor}\cite{ApostolovI}
A (pseudo-)K\"ahler structure $(g,J,\omega)$ is weakly Bochner-flat if and only if 
the normalized Ricci tensor~\eqref{eq:normRic} is a solution of Equation~\eqref{eq:main}.
\end{lem}
Lemma~\ref{lem:Bochnertensor} is the key which allows to apply the description of (pseudo-)K\"ahler manifolds
with hermitian solution of~\eqref{eq:main} from Theorem~\ref{thm:localclass} to the weakly Bochner-flat and, in particular, to the Bochner-flat case.

\subsection{K\"ahler-Einstein and weakly Bochner-flat metrics}

The next theorem will be proven in \S\ref{sec:WBF}. It is the first step in the proof of Theorem~\ref{thm:Bochner0}.

\begin{thm}[Weakly Bochner-flat (pseudo-)K\"ahler structures]\label{thm:WBF}
The (pseudo-)K\"ahler structure $(g,J,\omega)$ from~\eqref{eq:gjwrhocoord} is weakly-Bochner flat if
\begin{enumerate}
\item[(a)] $H(t)=\Theta_j'(t) + \Theta_j(t) \frac{p'_\c(t)}{p_\c(t)}$ is a polynomial of degree $\leq \ell+1$ independent of $j$,
where $p_\c(t)=\det_\C(t\cdot\Id-A_\c)$ is the characteristic polynomial of $A_\c$ considered as complex endomorphism 
(and $\Theta_j'(t)=\tfrac{\d}{\d t}\Theta_j(t)$, $p_\c'(t)=\tfrac{\d}{\d t}p_\c(t)$), and
\item[(b)] $\mathrm{Ric}_\c=-\frac{1}{2}g_\c(H(A_\c)\cdot,\cdot)$, where $\mathrm{Ric}_\c$ is the Ricci tensor of $(g_\c,\omega_\c)$.
\end{enumerate} 
Moreover, the (pseudo-)K\"ahler structure $(g,J,\omega)$ from~\eqref{eq:gjwrhocoord} is K\"ahler-Einstein if and only if in addition to (a) and (b) the polynomial
$H(t)$ has degree $\leq \ell$ and it is Ricci-flat if and only if the degree of $H(t)$ is $\leq \ell-1$.

Conversely, every $2n$-dimensional weakly Bochner-flat (pseudo-)K\"ahler manifold $(M,g,J,\omega)$ takes locally, in a neighbourhood of a generic point, 
the form~\eqref{eq:gjwrhocoord} with (a) and (b) satisfied.
\end{thm}

\begin{rem}[The statement of Theorem~\ref{thm:WBF} for $\ell=0$]
Consider the (pseudo-)K\"ahler structure $(g,J,\omega)$ from~\eqref{eq:gjwrhocoord} for $\ell=0$. 
Note that the condition $\ell=0$ is equivalent to $A$ from~\eqref{eq:Arhocoord} being parallel. 
In this case, $(g,J,\omega,A)$ is just given by the constant block: $g=g_\c$, $J=J_\c$, $\omega=\omega_\c$ and $A=A_\c$.
The first part of Theorem~\ref{thm:WBF} states that for $\ell=0$, $(g,J,\omega)$ from~\eqref{eq:gjwrhocoord} is weakly Bochner-flat if
$\Ric=\Ric_\c$ is parallel. This statement is already implied by Lemma~\ref{lem:Bochnertensor}: if the Ricci-tensor of a (pseudo-)K\"ahler manifold
is parallel, then also the normalized Ricci tensor is parallel and is therefore a solution of Equation~\eqref{eq:main}.

The converse direction of Theorem~\ref{thm:WBF} still remains valid for the description of weakly Bochner-flat manifolds
with parallel (normalized) Ricci tensor, although the statement in this case is empty: obviously, such manifolds 
are given by the formulas~\eqref{eq:gjwrhocoord} with $\ell=0$ and $\Ric_\c$ parallel.
\end{rem}

\begin{rem}[Relation to the results of \cite{ApostolovI}]
This statement generalizes the description of positive definite weakly Bochner-flat metrics from \cite{ApostolovI}:
suppose all eigenvalues $c_1,\dots,c_N$ of $A_\c$ are real and $A_\c$ is diagonalizable.
Decomposing $(g_\c,\omega_\c)$ and $A_\c$ as in \eqref{eq:deRhamdecomp} w.r.t.\ the parallel eigenspace distributions of $A_\c$,
we obtain from Theorem~\ref{thm:WBF} that the Ricci tensor $\mathrm{Ric}_\gamma$ of each component $(g_\gamma,\omega_\gamma)$
satisfies $\mathrm{Ric}_\gamma=-\tfrac{1}{2}H(c_\gamma)g_\gamma$, $1\leq \gamma\leq N$.
Thus, each component $(g_\gamma,\omega_\gamma)$ is K\"ahler-Einstein with scalar curvature $\mathrm{Scal}_\gamma=-n_\gamma H(c_\gamma)$,
where $n_\gamma$ is the multiplicity of $c_\gamma$ in the characteristic polynomial $p_\c(t)$ of $A_\c$.
\end{rem}

\begin{rem}[Indecomposable non-Einstein (pseudo-)K\"ahler metrics with parallel Ricci]\label{rem:nonEinstein}
In contrast to the positive definite case, the constant block $(g_\c,\omega_\c)$ 
in Theorem~\ref{thm:WBF} is in general not a direct product of K\"ahler-Einstein metrics.
For instance, using Theorem~\ref{thm:symmetricspace} and Remark~\ref{rem:symmetricspace}, it is easy to construct
(symmetric) examples of indecomposable non-Einstein (pseudo-)K\"ahler metrics with parallel Ricci tensor:
use as the data in Theorem~\ref{thm:symmetricspace} the vector space $\mathbb V=\R^4$ with (pseudo-)hermitian structure $(g_0,J_0)$
and hermitian endomorphism $A_0$ as in Example~\ref{ex:symmetric} but this time with the quadratic polynomial $\Theta(t)=t^2$. 
The algebraic curvature tensor $R_0=R_{\Theta,A_0}$ (computed from~\eqref{eq:defRThA}) then has Ricci tensor $\Ric(R_0)=-2A_0$ and is not Einstein.
\end{rem}

\subsection{Bochner-flat metrics}
\label{subsec:algprelim}

Let us give a precise definition of the algebraic curvature tensors $R_{\Theta,A_0}$ appearing in Theorem~\ref{thm:symmetricspace}:
consider a (pseudo-)hermitian vector space $(\mathbb{V},g,J)$ of real dimension $2n$ with hermitian form $\omega=g(J\cdot,\cdot)$. 
For hermitian endomorphisms $A,B:\mathbb{V}\rightarrow \mathbb{V}$, we write for short $gA=g(A\cdot,\cdot)$, $gB=g(B\cdot,\cdot)$ and $\omega A=\omega(A\cdot,\cdot)$, 
$\omega B=\omega(B\cdot,\cdot)$ and we define 
\begin{equation}\label{eq:RAB}
R_{A,B}=gA\owedge gB+\omega A\owedge \omega B +2\omega A\odot \omega B\,
\end{equation}
which is an algebraic curvature tensor having K\"ahler symmetry. 
If it is not clear from the context which scalar product (compatible with a fixed complex structure $J$)
enters the construction of the tensors $R_{A,B}$, we will write $R_{g,A,B}$ instead of $R_{A,B}$ to avoid ambiguities.
It can be computed straightforwardly that the Ricci tensor and the scalar curvature corresponding to $R_{A,B}$ are given by 
\begin{equation}\label{eq:RicRAB}
\Ric(R_{A,B})=\tr(A)B+\tr(B)A+2(AB+BA),
\end{equation}
\begin{equation}\label{eq:ScalRAB}
\Scal(R_{A,B})=2\bigl(\tr(A)\tr(B)+2\tr(AB) \bigr).
\end{equation}
Note that for $A=B=\Id$, the tensor $R_{\Id,\Id}$ is the K\"ahler curvature tensor of constant holomorphic sectional curvature 
with scalar curvature $\Scal(R_{\Id,\Id})=8n(n+1)$. For a polynomial $\Theta(t)=\sum_{k=0}^N a_k t^k$ (with real coefficients) and 
a hermitian endomorphism $A:\mathbb V\rightarrow \mathbb V$, we now 
define the algebraic curvature tensor
\begin{equation}
\label{eq:defRThA}
R_{\Theta,A}=-\frac{1}{8}\sum_{k=1}^N a_k \sum_{r+s=k-1}R_{A^r,A^s}.
\end{equation}
Again, if there is some ambiguity in the choice of the scalar product, we write $R_{g,\Theta,A}$ instead of $R_{\Theta,A}$.
We see immediately that if $A=c\cdot\Id$ for some $c\in \R$, then $R_{\Theta,c\cdot\Id}=-\tfrac{1}{8}\Theta'(c)R_{\Id,\Id}$ 
is a tensor of constant holomorphic sectional curvature with scalar curvature $\Scal(R_{\Theta,c\cdot\Id})=-n(n+1)\Theta'(c)$.
We will explain in \S\ref{subsec:algprelim2} that $R_{\Theta,A}$ defined in \eqref{eq:defRThA} indeed arises
from the Riemannian curvature tensor $\tilde{R}_{\Theta,A}$ defined in~\eqref{eq:derivofpol} via projecting the latter onto
the space of K\"ahler curvature tensors. The proof of Theorem~\ref{thm:Bochner0} will be given in \S\ref{sec:Bochner}. 
The statement we will actually prove is Theorem~\ref{thm:Bochner} which in particular implies Theorem~\ref{thm:Bochner0} 
(cf.\ also Remarks~\ref{rem:reltoTHMsymmetricspace} and~\ref{rem:BFsymmetricspace}):

\begin{thm}[Bochner-flat (pseudo-)K\"ahler structures]\label{thm:Bochner}
The (pseudo-)K\"ahler structure $(g,J,\omega)$ from~\eqref{eq:gjwrhocoord} is Bochner-flat if 
\begin{enumerate}
\item[(a)] $\Theta_i(t)=\Theta(t)$ is a polynomial of degree $\leq \ell+2$ independent of $i$,
\item[(b)] $\Theta(A_\c)=0$ and 
\item[(c)] $R_\c=R_{g_\c,\Theta,A_\c}$, where $R_\c$ is the curvature tensor of $(g_\c,\omega_\c)$.
\end{enumerate} 
Under these conditions, when $\Theta(t)=C_2t^{\ell+2}+C_1t^{\ell+1}+\dots$, the curvature
tensor $R$ of $g$ is given by $R=R_{g,q,A}$, where $q(t)=C_2t^2+(C_2\mu_1+C_1)t$ and $A$ is the endomorphism from~\eqref{eq:Arhocoord}. 
Thus, $(g,J,\omega)$ has constant holomorphic sectional curvature with scalar curvature $\Scal=-n(n+1)C_1$ 
if and only if $C_2=0$ and it is flat if and only if in addition $C_1=0$.

Conversely, every $2n$-dimensional Bochner-flat (pseudo-)K\"ahler manifold $(M,g,J,\omega)$ 
takes locally, in a neighbourhood of a generic point, the form~\eqref{eq:gjwrhocoord} with (a)--(c) satisfied.
\end{thm}

\begin{rem}[Relation to the results of \cite{ApostolovI}]\label{rem:posdef}
This statement generalizes the description of positive definite Bochner-flat metrics from \cite{ApostolovI,Bryant}:
if $A_\c$ has only real eigenvalues $c_1,\dots,c_N$ and is diagonalizable, the condition $\Theta(A_\c)=0$ is equivalent to 
$\Theta(c_i)=0$ for $i=1,\dots,N$. 
Moreover, the condition $R_\c=R_{\Theta,A_\c}$ factorizes (cf.\ Proposition~\ref{prop:curvminpol}(a))
w.r.t.\ the decomposition \eqref{eq:deRhamdecomp} such that for each component 
$(g_\gamma,\omega_\gamma)$ of $(g_\c,\omega_\c)$ on which $A_\c=c_\gamma \Id$, it gives
$R_\gamma=-\frac{\Theta'(c_\gamma)}{8}R_{\Id,\Id}$, where $R_\gamma$ is the curvature tensor of $(g_\gamma,\omega_\gamma)$. Thus, 
$(g_\gamma,\omega_\gamma)$ has constant holomorphic sectional curvature with scalar curvature 
equal to $\mathrm{Scal}_\gamma=-\Theta'(c_\gamma)n_\gamma(n_\gamma+1)$ (cf.\ for instance~\eqref{eq:ScalRAB}), 
where $n_\gamma$ is the multiplicity of $c_\gamma$ in the characteristic polynomial $p_\c(t)$ of $A_\c$. 
\end{rem}

\begin{rem}[Relation to the weakly Bochner-flat case]\label{rem:relationtoWBF}
Of course the conditions $\Theta(A_\c)=0$ and $R_\c=R_{\Theta,A_\c}$ from Theorem~\ref{thm:Bochner} imply
condition $\mathrm{Ric}_\c=-\tfrac{1}{2}H(A_\c)$ from Theorem~\ref{thm:WBF} (as it should be since Bochner-flat
implies weakly Bochner-flat). This follows from Proposition~\ref{prop:curvminpol}(b). 
\end{rem}

\begin{rem}[Relation to Theorem~\ref{thm:symmetricspace}]\label{rem:reltoTHMsymmetricspace}
Since $A_\c$ is parallel w.r.t.\ $g_\c$ (cf.\ Theorem~\ref{thm:localclass}) 
and $R_\c=R_{g_\c,\Theta,A_\c}$ for the polynomial $\Theta(t)$ with constant coefficients (cf.\ Theorem~\ref{thm:Bochner}), 
the curvature tensor $R_\c$ of $g_\c$ must be parallel as well, hence, the constant block $(S,g_\c,J_\c,\omega_\c)$ 
is a locally symmetric space. Theorem~\ref{thm:symmetricspace} then
states that there are no more obstructions on $(S,g_\c,J_\c,\omega_\c)$ or $\Theta(t)$ arising from this condition
and therefore $(S,g_\c,J_\c,\omega_\c)$ is indeed of the form $\mathcal{S}_{g_0,J_0,A_0,\Theta}$, where $g_0=g(p),J_0=J(p)$ and $A_0=A(p)$
for some point $p\in S$.
\end{rem}

\begin{rem}[The locally symmetric case]\label{rem:BFsymmetricspace}
Clearly, a Bochner-flat (pseudo-)K\"ahler manifold $(M,g,J,\omega)$ is locally symmetric, that is, the curvature tensor $R$ is parallel,
if and only if its Ricci tensor is parallel, cf.\ Equation~\eqref{eq:Bochner}.
From the formula $R=R_{g,q,A}$ from Theorem~\ref{thm:Bochner} together with \eqref{eq:RicRAB}--\eqref{eq:defRThA},
we compute that the Ricci tensor and scalar curvature of $g$ are given by
$$
\Ric=-\frac{1}{8}[4(n+2)C_2A+2(n+2)C_2\tr(A)\Id+C\cdot\Id],
$$
\begin{equation}\label{eq:Scalg}
\Scal=-\frac{1}{2}(n+1)(n+2)C_2\tr(A)-\frac{n}{4}C=-(n+1)(n+2)C_2\mu_1+\tilde C
\end{equation}
for certain constants $C,\tilde C$. Thus, if $C_2\neq 0$, that is, $(g,J,\omega)$ has nonconstant holomorphic sectional curvature, then,
$(g,J,\omega)$ is locally symmetric if and only if $A$ is parallel. 
On the other hand, a solution $A$ of~\eqref{eq:main} is parallel if and only if $\ell=0$, i.e.,
all of its eigenvalues are constant (indeed, under this condition the vector field $\Lambda$ in~\eqref{eq:main}
vanishes identically). We obtain the statement from Theorem~\ref{thm:Bochner0} that the Bochner-flat (pseudo-)K\"ahler structure $(g,\omega)$
from~\eqref{eq:gwBochner} is locally symmetric if and only if either $C_2=0$ or $\ell=0$.
\end{rem}

\begin{rem}[The scalar curvature]
Suppose a (pseudo-)K\"ahler structure $(g,J,\omega)$ is Bochner-flat and of nonconstant holomorphic sectional curvature. 
From \eqref{eq:Scalg} (with $C_2\neq 0$), we see that the scalar curvature is constant if and only if $\ell=0$, that is,
if and only if $(g,J,\omega)$ is locally symmetric (cf. Remark~\ref{rem:BFsymmetricspace}).
This is the (pseudo-)K\"ahler version of a result from~\cite{Matsumoto0} obtained in the positive definite case. 
\end{rem}

\section{Proof of Theorem \ref{thm:WBF}}
\label{sec:WBF}

\subsection{The Ricci tensor of the (pseudo-)K\"ahler structure from \eqref{eq:gjwrhocoord} }
\label{subsec:Ricci}

Let $r(X,Y)=\mathrm{Ric}(JX,Y)$ be the Ricci form of a K\"ahler structure 
$(g,J,\omega)$. For the K\"ahler structure from \eqref{eq:gjwrhocoord} 
a Ricci potential $\kappa$ such that $\d\d^\c\kappa=r$ (where for a function $f$, we define $\d^\c f=-\d f\circ J$)
can be computed in the same way as for positive definite $g$. This has already been done in \cite[Section 5.1]{ApostolovI}
(again the proof holds in any signature): if $\kappa_\c$ is a Ricci potential for $(g_\c,\omega_\c)$, 
then 
\begin{equation}
\label{eq:Riccipotential}
\kappa=\kappa_\c-\frac{1}{2}\log|\prod_{i=1}^\ell \Theta_i(\rho_i) p_\c (\rho_i)|,
\end{equation}
where $p_\c(t)=\det_\C(t\cdot \Id-A_\c)$. 
From this it is straight-forward to compute $\mathrm{Ric}=\d\d^\c\kappa(\cdot,J\cdot)$: define functions 
$$
H_i(t)=\Theta'_i(t) + \Theta_i(t) \frac{p'_\c(t)}{p_\c(t)}, \quad1\leq i\leq \ell,
$$
where $p'_\c(t)=\tfrac{\d}{\d t}p_\c(t)$. Then we obtain
\begin{equation}\label{eq:Riccirhocoord}
\begin{array}{c}
\displaystyle \mathrm{Ric}=\sum_{i=1}^\ell \lambda_i\frac{\Delta_i}{\Theta_i}\d\rho_i^2 
+\sum_{i,j=1}^\ell \sum_{k=1}^\ell \mu_{i-1}(\hat{\rho_k})\mu_{j-1}(\hat{\rho_k})\frac{\lambda_k \Theta_k}{\Delta_k} 
\theta_i\theta_j\vspace{1mm}\\
\displaystyle+\mathrm{Ric}_\c+\frac{1}{2}\sum_{i=1}^\ell \frac{H_i}{\Delta_i}g_\c\bigl(p_\nc(A_\c)(A_\c-\rho_i\Id)^{-1}\cdot,\cdot\bigr),
\end{array}
\end{equation}
where $\mathrm{Ric}_\c$ is the Ricci tensor of $(g_\c,\omega_\c)$ and 
$$
\lambda_i=-\frac{1}{2}\frac{\D}{\D\rho_i}\sum_{j=1}^\ell\frac{H_j}{\Delta_j}.
%=-\frac{1}{2}\left(\frac{H_i'}{\Delta_i}-\sum_{j\neq i}\frac{1}{\rho_i-\rho_j}\left(\frac{H_i}{\Delta_i}+\frac{H_j}{\Delta_j}\right)\right)
$$
We see from \eqref{eq:gjwrhocoord} and \eqref{eq:Riccirhocoord} that $\lambda_i$ is an eigenvalue of $\Ric$ 
(when considered as endomorphism) with corresponding eigenspace $\mathrm{span}\{\D_{\rho_i},J\D_{\rho_i}\}$.

\subsection{Proof of Theorem \ref{thm:WBF}}
\label{subsec:WBF}

Consider a (pseudo-)K\"ahler manifold $(M,g,J,\omega)$ with a hermitian solution $A$ of \eqref{eq:main} having nonconstant eigenvalues 
(possibly complex-valued) $\rho_1,\dots,\rho_\ell$ as in~\eqref{eq:gjwrhocoord} and~\eqref{eq:Arhocoord}. 
Then $(g,J,\omega)$ is weakly Bochner-flat with normalized Ricci tensor $\tilde{\mathrm{Ric}}=\Ric-\tfrac{\Scal}{2(n+1)}\Id$
a (constant coefficient) linear combination of $A$ and $\Id$ if and only if
$$
\mathrm{Ric}=-\frac{c_2}{2} A+\left(\frac{\mathrm{Scal}}{2(n+1)}-\frac{\tilde{c}_1}{2}\right)\Id
$$
for certain $c_2,\tilde c_1\in \R$ (cf.\ Lemma~\ref{lem:Bochnertensor}). Taking the trace yields
$$
\frac{\mathrm{Scal}}{2(n+1)}=-\frac{c_2}{2} \mu_1-\frac{c_2}{4} \tr(A_\c)-\frac{n\tilde{c}_1}{2}.
$$
Inserting this into the formula for $\Ric$ and defining the constant $c_1=\tfrac{c_2}{2}\tr(A_\c)+(n+1)\tilde{c}_1$, we obtain
that $(g,J,\omega)$ is weakly Bochner-flat with normalized Ricci tensor a linear combination of $A$ and $\Id$ if and only if
\begin{equation}\label{eq:WBF}
\mathrm{Ric}=-\frac{1}{2}(c_2 A+(c_2\mu_1+c_1)\Id)
\end{equation}
for certain $c_1,c_2\in \R$. In particular, $(g,J,\omega)$ is K\"ahler-Einstein if and only if $c_2=0$
and $(g,J,\omega)$ is Ricci-flat if and only if in addition $c_1=0$. 
Comparing $A$ from \eqref{eq:Arhocoord} with $\mathrm{Ric}$ from \eqref{eq:Riccirhocoord}, 
the first set of equations we obtain from~\eqref{eq:WBF} is
$$
\frac{\D}{\D\rho_i}\left(\sum_{j=1}^\ell\frac{H_j(\rho_j)}{\Delta_j}\right)
=c_2(2\rho_i+\mu_1(\hat \rho_i))+c_1,\quad 1\leq i\leq \ell,
$$
which after integration (and using $\mu_1^2-\mu_2=\sum_{j=1}^\ell\rho_j^2+\mu_2$) yields

\begin{equation}\label{eq:solutionPDE}
\sum_{j=1}^\ell\frac{H_j(\rho_j)}{\Delta_j}
=c_2\Big( \sum_{j=1}^\ell \rho_j^2+\mu_2\Big)+c_1\mu_1+c_0.
\end{equation}
We recall some standard identities that will be used frequently throughout the article:

\begin{lem}\label{lem:VandId}
Consider free variables $\rho_1,\dots,\rho_\ell$ (which are either real or arise as complex conjugate pairs)
such that $\rho_i\neq \rho_j$ for $i\neq j$. Let $\Delta_i=\prod_{j\neq i}(\rho_i-\rho_j)$
and $p_\nc(t)=\prod_{i=1}^\ell(t-\rho_i)$.
\begin{enumerate}
\item[(a)] We have
\begin{equation}\label{eq:VandId}
\begin{array}{c}
\displaystyle\sum_{j=1}^\ell\frac{\rho_j^{\ell-s}}{\Delta_j}=0\quad(2\leq s\leq \ell),
\quad\sum_{j=1}^\ell\frac{\rho_j^{\ell-1}}{\Delta_j}=1,\quad\sum_{j=1}^\ell\frac{\rho_j^{\ell}}{\Delta_j}=\mu_1,
\quad
%\vspace{1mm}\\
%\displaystyle
\sum_{j=1}^\ell\frac{\rho_j^{\ell+1}}{\Delta_j}=\mu_1^2-\mu_2
\vspace{1mm}\\
\displaystyle
\mbox{and, more generally,}\quad
\sum_{j=1}^\ell\frac{\rho_j^{\ell+m}}{\Delta_j}=\sum_{j_1+\dots+j_\ell=m+1}\rho_1^{j_1}\dots\rho_\ell^{j_\ell}\quad\mbox{for}\quad m\geq -1.
%\mbox{ and }
%\sum_{j=1}^\ell\frac{\rho_j^{\ell+2}}{\Delta_j}=(\mu_1^2-\mu_2)\mu_1-\mu_1\mu_2+\mu_3.
\end{array}
\end{equation}
\item[(b)] Suppose we have $\sum_{i=1}^\ell\tfrac{f_i(\rho_i)}{\Delta_i}=h(\vec{\rho})$, 
where $f_i:\R\rightarrow \R$ are functions of one variable and $h:\R^\ell\rightarrow \R$ 
is a polynomial in $\rho_1,\dots,\rho_\ell$ of degree $\leq k$.  
Then, $f(t)=f_i(t)$ is a polynomial of degree $\leq \ell+k-1$ (independent of $i$). 
\item[(c)] Let $\Theta(t)=C_2t^{\ell+2}+C_1t^{\ell+1}+C_0 t^\ell+\dots$ be a polynomial of degree $\leq \ell+2$.  
Then, we have the identity
\begin{equation}\label{eq:ApplVandIdAppl}
\sum_{i=1}^\ell \frac{\Theta(\rho_i)}{\Delta_i(\rho_i-t)}
=-\frac{\Theta(t)}{p_\nc(t)}+C_2t^2+(C_2\mu_1+C_1)t+C_2(\mu_1^2-\mu_2)+C_1\mu_1+C_0.
\end{equation}
\end{enumerate}
\end{lem}

\begin{proof}
Part (a) is taken from \cite[Appendix B]{ApostolovI}. To show part (b) we first notice that 
\begin{equation}
\label{eq:frombols1}
\sum_{i=1}^\ell\frac{\D}{\D\rho_i}\sum_{j=1}^\ell\frac{f_j(\rho_j)}{\Delta_j}=\sum_{j=1}^\ell\frac{f_j'(\rho_j)}{\Delta_j}.
\end{equation}
Iteratively we conclude
$$
\sum_{i_1,\dots,i_{k+1}=1}^\ell\frac{\D}{\D\rho_{i_1}}\dots \frac{\D}{\D\rho_{i_{k+1}}}
\sum_{j=1}^\ell\frac{f_j(\rho_j)}{\Delta_j}=\sum_{j=1}^\ell\frac{f_j^{(k+1)}(\rho_j)}{\Delta_j}.
$$
Applying this to $\sum_{i=1}^\ell\tfrac{f_i(\rho_i)}{\Delta_i}=h(\vec{\rho})$ yields $\sum_{i=1}^\ell\tfrac{f_i^{(k+1)}(\rho_i)}{\Delta_i}=0$.
The claim now follows from part (a) or \cite[Lemma 5.12]{BMR}.

\medskip
To show part (c) we apply the Vandermonde identities~\eqref{eq:VandId} 
in the $\ell+1$ variables $\rho_1,\dots,\rho_\ell,t$ to the degree $\leq \ell+2$
polynomial $\Theta(t)$. We obtain
$$
\sum_{i=1}^\ell \frac{\Theta(\rho_i)}{\Delta_i(\rho_i-t)}+\frac{\Theta(t)}{\prod_{i=1}^\ell (t-\rho_i)}=C_2\bigl((\mu_1+t)^2-\mu_1 t-\mu_2\bigr)+C_1(\mu_1+t)+C_0
$$
$$
=C_2t^2+(C_2\mu_1+C_1)t+C_2(\mu_1^2-\mu_2)+C_1\mu_1+C_0
$$
as we claimed.
\end{proof}
Applying the identities \eqref{eq:VandId} to \eqref{eq:solutionPDE}, we see that $H_j(t)-c_2 t^{\ell+1}-c_1t^{\ell}-c_0t^{\ell-1}$ is a polynomial
of degree $\leq \ell-2$ in $t$ (with constant coefficients independent of $j$). We write $H(t)=H_j(t)$ for the polynomial
of degree $\leq \ell+1$.

Now let us turn to the implications for the constant block $(g_\c,\omega_\c)$ of the K\"ahler structure from \eqref{eq:gjwrhocoord}
arising from the weakly Bochner-flat condition~\eqref{eq:WBF}: combining~\eqref{eq:WBF} with \eqref{eq:gjwrhocoord}, \eqref{eq:Arhocoord}
and~\eqref{eq:Riccirhocoord} yields
\begin{equation}\label{eq:horizcomp}
-2g_\c^{-1}\mathrm{Ric}_\c=\sum_{i=1}^\ell \frac{H(\rho_i)}{\Delta_i}p_\nc(A_\c)(A_\c-\rho_i\Id)^{-1}
+c_2 p_\nc(A_\c)A_\c+\Big(c_2\sum_{i=1}^\ell \rho_i+c_1\Big)p_\nc(A_\c).
\end{equation}
Thus, $-2g_\c^{-1}\mathrm{Ric}_\c=p(A_\c)$ for some polynomial $p(t)$ with coefficients that may a priori depend on $\vec\rho$.
To show that the polynomial has actually constant coefficients, fix $\rho_1<\dots<\rho_\ell$ and choose $t\in \R$
such that $t\neq \rho_i$, $1\leq i\leq \ell$. Then we see from \eqref{eq:horizcomp} that 
$$
\frac{p(t)}{p_\nc(t)}=-\sum_{i=1}^\ell \frac{H(\rho_i)}{\Delta_i(\rho_i-t)}
+c_2\Big(\sum_{i=1}^\ell \rho_i+ t\Big)+c_1.
$$
Recall that $H(t)=c_2t^{\ell+1}+c_1t^{\ell}+c_0t^{\ell-1}+\dots$ is a polynomial of degree $\leq \ell+1$.
Then the identities \eqref{eq:VandId} taken in the $\ell+1$ variables $\rho_1,\dots,\rho_\ell,t$ show that
$$
\sum_{i=1}^\ell \frac{H(\rho_i)}{\Delta_i(\rho_i-t)}+\frac{H(t)}{p_\nc(t)}=c_2
\Big(\sum_{i=1}^\ell\rho_i+t\Big)+c_1.
$$
Inserting this into the last equation, we see that $p(t)=H(t)$ holds for all $t\in \R$ 
and therefore \eqref{eq:horizcomp} is equivalent to
$$
\mathrm{Ric}_\c=-\frac{1}{2}g_\c\big(H(A_\c)\cdot,\cdot\big).
$$
This completes the proof of Theorem \ref{thm:WBF}.

\section{Proof of Theorem \ref{thm:Bochner}}
\label{sec:Bochner}

\subsection{Algebraic preliminaries}
\label{subsec:algprelim2}

We continue to work with the notation introduced in \S\ref{subsec:algprelim}.
Let $(\mathbb{V},g,J)$ be a (pseudo-)hermitian vector space with hermitian form $\omega=g(J\cdot,\cdot)$. 
Denote by 
$$
\mathfrak{so}(g)=\big\{X\in \mathfrak{gl}(\mathbb{V}):g(X\cdot,\cdot)=-g(\cdot,X\cdot)\big\},\quad
\mathfrak{u}(g,J)=\big\{X\in \mathfrak{so}(g):XJ=JX\big\}
$$
the Lie algebras of skew-symmetric and skew-hermitian endomorphisms of $\mathbb V$ respectively.
Consider the tensor space $\bigotimes^4 \mathbb{V}^*$ and the symmetry relations
\begin{itemize}
\item[($\mathcal{K}1$)] $R(x,y,u,v)=-R(y,x,u,v)=-R(x,y,v,u)$,
\item[($\mathcal{K}2$)] $R(x,y,u,v)=R(u,v,x,y)$,
\item[($\mathcal{K}3$)] $R(x,y,u,v)+R(x,v,y,u)+R(x,u,v,y)=0$,
\item[($\mathcal{K}4$)] $R(x,y,u,v)=R(Jx,Jy,u,v)=R(x,y,Ju,Jv)$.
\end{itemize}
Let $\mathcal{R}(\mathbb V)$ be the space of algebraic curvature tensors on $\mathbb V$
given by all elements of $\bigotimes^4 \mathbb{V}^*$ satisfying ($\mathcal{K}1$)--($\mathcal{K}3$).
The subspace $\mathcal{K}(\mathbb V)\subseteq \mathcal{R}(\mathbb V)$ of elements which satisfy in addition
($\mathcal{K}4$) are the curvature tensors with K\"ahler symmetry.
Each $R\in \bigotimes^4 \mathbb{V}^*$ satisfying ($\mathcal{K}1$) 
can be identified with a linear operator ${\mathsf R}: \mathfrak{so}(g)\rightarrow \mathfrak{so}(g)$.
Indeed, if we adopt the $(1,3)$-tensor notation for $R$, i.e., $R(x,y,u,v)=g(x,R(u,v)y)$,  
and identify $\Lambda^2 \mathbb V$ with $\mathfrak{so}(g)$ by 
$$
u\wedge v= u^\flat\otimes v-v^\flat\otimes u,
$$
then we can set by definition:
$$
{\mathsf R}(u\wedge v)=2R(u,v).
$$
Similarly,  if we consider generating elements 
$$
u\wedge_J v=u\wedge v+Ju\wedge Jv
$$
of $\mathfrak{u}(g,J)$, then we can view an element 
$R\in \bigotimes^4 \mathbb{V}^*$ satisfying ($\mathcal{K}1$) and ($\mathcal{K}4$)
as a linear operator $\mathsf R: \mathfrak{u}(g,J)\rightarrow \mathfrak{u}(g,J)$, namely: 
$$
{\mathsf R}(u\wedge_J v)=4R(u,v).
$$ 
Recall the formula for the orthogonal projection $\mathcal{R}(\mathbb V)\rightarrow \mathcal{K}(\mathbb V)$
(cf.\ for instance~\cite{Vanhecke}):

\begin{lem}\label{lem:projection}
The linear mapping ${\mathsf{pr}}:\mathcal{R}(\mathbb V)\rightarrow \mathcal{R}(\mathbb V)$,
$$
{\mathsf{pr}}(R)(x,y,u,v)=\frac{1}{16}\Big[3R(x,y,u,v)+3R(Jx,Jy,u,v)+3R(x,y,Ju,Jv)+3R(Jx,Jy,Ju,Jv)
$$
$$
+R(x,Ju,v,Jy)-R(x,v,Jy,Ju)+R(x,Jv,Jy,u)-R(x,u,Jv,Jy)
$$
$$
+R(Jx,u,Jv,y)-R(Jx,Jv,y,u)+R(Jx,v,y,Ju)-R(Jx,Ju,v,y)\Big],
$$
satisfies ${\mathsf{pr}}^2={\mathsf{pr}}$ and ${\mathsf{pr}}(\mathcal{R}(\mathbb V))=\mathcal{K}(\mathbb V)$.
\end{lem}
We now use the projection $\mathsf{pr}$  to interpret the curvature tensors \eqref{eq:RAB} and  \eqref{eq:defRThA} defined above in 
\S\ref{subsec:algprelim} as operators on $\mathfrak{u}(g,J)$.
Let $A,B:\mathbb V\rightarrow \mathbb V$ be hermitian endomorphisms. 
It is easy to check that the linear operator ${\mathsf R}_{A,B}: \mathfrak{u}(g,J) \to \mathfrak{u}(g,J)$ 
corresponding to the curvature tensor $R_{A,B}\in \mathcal{K}(\mathbb V)$ defined in \eqref{eq:RAB} is given by
\begin{equation}
\label{eq:RABopK}
{\mathsf R}_{A,B}(X)=-4(AXB +BXA-\tfrac{1}{2}\tr(JBX)JA-\tfrac{1}{2}\tr(JAX)JB),\quad X\in \mathfrak{u}(g,J).
\end{equation}
Moreover, we have
\begin{equation}
\label{eq:pr1}
{\mathsf R}_{A,B}=2\, \mathsf{pr}\left(  \widetilde {\mathsf R}_{A,B}  \right),
\end{equation} 
where $\widetilde {\mathsf R}_{A,B}:  \mathfrak{so}(g) \to \mathfrak{so}(g)$, $\widetilde{\mathsf R}_{A,B} \in \mathcal{R}(\mathbb V)$, is defined by
\begin{equation}
\label{eq:RABtilde}
\widetilde{\mathsf R}_{A,B} (X) = -4 (AXB + BXA),  \quad X\in \mathfrak{so}(g).
\end{equation} 
Now let $q(t)=\sum_{k=0}^N a_k t^k$ be any polynomial. From the definition~\eqref{eq:defRThA} and Equation~\eqref{eq:RABopK}, it follows that 
the linear operator ${\mathsf R}_{q,A}: \mathfrak{u}(g,J) \to \mathfrak{u}(g,J)$ corresponding to
the curvature tensor $R_{q,A}\in \mathcal{K}(\mathbb V)$ takes the form
\begin{equation}
\label{eq:polcurvop}
{\mathsf R}_{q,A} (X) = -\frac{1}{8}\sum_{k=1}^N a_k  \!\! \sum_{r+s=k-1} \!\! {\mathsf R}_{A^r,A^s}(X) = 
\sum_{k=1}^N a_k \!\! \sum_{r+s=k-1} \!\! \big(A^s X A^r  - \tfrac{1}{2}\tr(JA^sX)JA^r\big)
\end{equation}
for $X\in \mathfrak{u}(g,J)$. By \eqref{eq:pr1} we have
\begin{equation}
\label{eq:pr2}
{\mathsf R}_{q,A} = 2\,{\mathsf {pr}}\left(\widetilde{\mathsf R}_{q,A}\right),
\end{equation}
where $\widetilde{\mathsf R}_{q,A} : \mathfrak{so}(g) \to \mathfrak{so}(g)$, $\widetilde{\mathsf R}_{q,A} \in \mathcal{R}(\mathbb V)$, is defined by
\begin{equation}
\label{eq:polcurvopRiem0}
\widetilde{\mathsf R}_{q,A} (X) = -\frac{1}{8}\sum_{k=1}^N a_k\sum_{r+s=k-1}\widetilde{\mathsf R}_{A^r,A^s}(X) = \sum_{k=1}^N a_k \sum_{r+s=k-1} A^s X A^r.
\end{equation}
The operator $\widetilde{\mathsf R}_{q,A}$ can be written more compactly in the form
\begin{equation}
\label{eq:polcurvopRiem}
\widetilde{\mathsf R}_{q,A}(X)=\frac{\d}{\d t}\Big|_{t=0}q(A+tX).
\end{equation}

\begin{rem}[Changing the polynomial]\label{rem:changepol}
The zero-order term $a_0$ of $q(t)$ does not contribute to ${\mathsf R}_{q,A}$ from \eqref{eq:polcurvop}. 
We also see that for any $\lambda\in \C$ (when we complexify $\mathbb V$
and extend all objects complex-linearly), 
${\mathsf R}_{q,A}={\mathsf R}_{\tilde q,A-\lambda\Id}$ for the modified polynomial
$\tilde{q}(t)=q(t+\lambda)$. This follows by applying the projection 
${\mathsf {pr}}:\mathcal R(\mathbb V)\rightarrow \mathcal K(\mathbb V)$
to the obvious identity $\widetilde{{\mathsf R}}_{q,A}=\widetilde{{\mathsf R}}_{\tilde q,A-\lambda\Id}$
for the curvature operators from \eqref{eq:polcurvopRiem}.
\end{rem}

We can generalize the above constructions by considering a rational or real analytic function $f:\R\rightarrow \R$
instead of a polynomial $q(t)$:
for $A:\mathbb V\rightarrow \mathbb V$ hermitian, define an operator $\widetilde{{\mathsf R}}_{f,A}:\mathfrak{so}(g)\rightarrow \mathfrak{so}(g)$
as in \eqref{eq:polcurvopRiem} by
$$
\widetilde{{\mathsf R}}_{f,A}(X)=\frac{\d }{\d t}\Big|_{t=0}f(A+tX).
$$
Since $f$ is real analytic we obtain as above that $\widetilde{{\mathsf R}}_{f,A}\in \mathcal{R}(\mathbb V)$. 
Then, we define 
$$
{\mathsf R}_{f,A}=2\, \mathsf{pr}\left(\widetilde{{\mathsf R}}_{f,A}\right)\in \mathcal{K}(\mathbb V),
$$  
where ${\mathsf{pr}}:\mathcal{R}(\mathbb V)\rightarrow \mathcal{K}(\mathbb V)$ is the projection from 
Lemma~\ref{lem:projection}. 

\begin{exmp}
\label{exmp:main}
Let $f(t)=(\rho-t)^{-1}$ for some $\rho\in \R$.
For $X\in \mathfrak{so}(g)$ we compute
$$
\widetilde{\mathsf R}_{f,A}(X)=\frac{\d}{\d t}\Big|_{t=0}(\rho\Id-A-tX)^{-1}
=(A-\rho\Id)^{-1}X(A-\rho\Id)^{-1}.
$$
Thus, by \eqref{eq:RABtilde} we have $\widetilde{\mathsf R}_{f,A}=-\tfrac{1}{8}\widetilde{\mathsf R}_{(A-\rho\Id)^{-1}, (A-\rho\Id)^{-1}}$
and therefore, by \eqref{eq:pr1}, 
$$
{\mathsf R}_{f,A}=2\, \mathsf{pr}\left(\widetilde{\mathsf R}_{f,A}\right)=-\tfrac{1}{8}{\mathsf R}_{(A-\rho\Id)^{-1},(A-\rho\Id)^{-1}}.
$$
\end{exmp}
If there are any ambiguities in the choice of the scalar product $g$ (compatible with a fixed complex structure $J$),
we write ${\mathsf R}_{g,A,B}$, $\widetilde{{\mathsf R}}_{g,A,B}$, ${\mathsf R}_{g,f,A}$ and $\widetilde{{\mathsf R}}_{g,f,A}$ 
instead of ${\mathsf R}_{A,B}$, $\widetilde{{\mathsf R}}_{A,B}$, ${\mathsf R}_{f,A}$ and $\widetilde{{\mathsf R}}_{f,A}$.
A similar convention is adopted for the corresponding algebraic curvature tensors.

\medskip
We will show next that the operators ${\mathsf R}_{q,A}$ for fixed $A$ arise as solutions of a certain
inhomogeneous linear equation.
The first statement of the next lemma also appeared in \cite[Remark A.1]{BMR}. 
The second statement appeared in the proof of \cite[Proposition A.2]{BMR}.

\begin{lem}\label{lem:curvpol}
Let $(\mathbb V,g,J)$ be a (pseudo-)hermitian vector space
and let $A,B:\mathbb V\rightarrow \mathbb V$ be hermitian endomorphisms.
\begin{enumerate}
\item[(a)] Suppose an operator ${\mathsf R}:\mathfrak{u}(g,J)\rightarrow \mathfrak{u}(g,J)$ satisfies 
\begin{equation}\label{eq:secop0}
[{\mathsf R}(X),A]=[X,B]\quad\mbox{for all}\quad X\in \mathfrak{u}(g,J).
\end{equation}
Then, $B=q(A)$ for a certain polynomial $q(t)$ with real coefficients
(unique up to adding a polynomial $r(t)$ such that $r(A)=0$).
\item[(b)] Let $q(t)$ be a polynomial with real coefficients. Consider the equation
\begin{equation}\label{eq:secop}
[{\mathsf R}(X),A]=[X,q(A)]\quad\mbox{for all}\quad X\in \mathfrak{u}(g,J)
\end{equation}
on operators ${\mathsf R}:\mathfrak{u}(g,J)\rightarrow \mathfrak{u}(g,J)$. A particular solution is given by ${\mathsf R}={\mathsf R}_{q,A}\in \mathcal{K}(\mathbb V)$.
Any solution of \eqref{eq:secop} can be written as ${\mathsf R}_{q,A}+{\upvarphi}$ for an operator ${\upvarphi}:\mathfrak{u}(g,J)\rightarrow \mathfrak{u}(g,J)$
satisfying $[\upvarphi(X),A]=0$ for all $X\in \mathfrak{u}(g,J)$.
\item[(c)]  The Ricci tensor of the operator ${\mathsf R}_{q,A}$ (resp.\ the Ricci tensor of the curvature tensor $R_{q,A}$ corresponding to ${\mathsf R}_{q,A}$)
takes the following explicit form:
\begin{equation}
\label{eq:RicRTheta2}
\Ric \left({\mathsf R}_{q,A}\right) = -\frac{1}{2} \left( q'(A) + \sum_{\lambda_i \in \mathrm{Spec}(A)} q_{\lambda_i} (A)  \right),
\end{equation}
where the polynomial $q_{\lambda_i}(t)$ is defined from the relation $q(t) - q(\lambda_i) = q_{\lambda_i} (t) (t-\lambda_i)$ and $q'(t)=\tfrac{\d}{\d t}q(t)$.
\end{enumerate}
\end{lem}

\begin{rem}[Sectional operators]
The commutator relation \eqref{eq:secop0} plays an important role in the theory of 
integrable systems on Lie algebras \cite{BFom70,Man,Mish} -- note that \eqref{eq:secop0} also makes sense 
on an arbitrary semisimple Lie algebra $\mathfrak{g}$. Solutions ${\mathsf R}:\mathfrak{g}\rightarrow \mathfrak{g}$ 
are known as sectional operators. 
The analogue of~\eqref{eq:secop0} for $\mathfrak{g}=\mathfrak{so}(g)$ also appeared in (real) 
projective geometry~\cite[\S 3]{Fubini} and in the construction of certain holonomy algebras for (pseudo-)Riemannian metrics \cite{BolTsonev}.
The properties of solutions of \eqref{eq:secop0} have been applied to c-projective geometry in~\cite[Appendix A]{BMR}
and we will make further use of it in \S\ref{subsec:BochnerflatTheta} and \S\ref{subsec:Bochnerflatconst} below.  
\end{rem}

\begin{proof}[Proof of Lemma \ref{lem:curvpol}]
(a) Consider $\mathbb V$ as a complex vector space (with multiplication $(a+ib)x=ax+bJx$). 
For an arbitrary  complex-linear endomorphism $Y:\mathbb V\rightarrow \mathbb V$ and 
$X\in \mathfrak{u}(g,J)$ we compute
$$
\tr_{\C} ( X \cdot[B, Y]) = \tr_{\C} ( Y\cdot [X, B] ) = \tr_{\C} ( Y\cdot [{\mathsf R}(X), A]) = \tr_{\C} ( {\mathsf R}(X)\cdot[A,Y]),
$$
 where $\tr_\C$ denotes the complex trace. Since $\mathfrak u(g,J)$ spans $\mathfrak {gl}_\C (\mathbb V)$ in the complex
 sense, we conclude that $[B, Y]=0$ for any $Y$ that commutes with $A$. It is a well-known algebraic fact that in this case $B$ 
can be written as a polynomial of $A$. Moreover, as both $A$ and $B$ are hermitian,  this polynomial must have real coefficients.

(b) For fixed $A,q$, \eqref{eq:secop} is an inhomogeneous linear equation on linear operators 
${\mathsf R}:\mathfrak{u}(g,J)\rightarrow \mathfrak{u}(g,J)$. A particular solution is given by
$$
{\mathsf R}_0(X)=\frac{\d}{\d t}\Big|_{t=0}q(A+tX)
$$
which can be seen by taking the $t$-derivative at $t=0$ of the identity $0=[q(A+tX),A+tX]$.
We see that ${\mathsf R}_0=\widetilde{{\mathsf R}}_{q,A}|_{\mathfrak{u}(g,J)}$, where $\widetilde{{\mathsf R}}_{q,A}\in \mathcal{R}(\mathbb V)$
is defined in~\eqref{eq:polcurvopRiem0}.
A comparison with \eqref{eq:polcurvop} shows that ${\mathsf R}_{q,A}(X)$ and ${\mathsf R}_0(X)$
only differ by terms that commute with $A$. Therefore also ${\mathsf R}_{q,A}$ solves \eqref{eq:secop}.
Then any solution of the inhomogeneous equation \eqref{eq:secop} can be written as the sum of 
the particular solution ${\mathsf R}_{q,A}$ and a solution $\upvarphi$ of the homogeneous equation $[\upvarphi(X),A]=0$.

(c)  Since formula \eqref{eq:RicRTheta2} is linear in $q$, it is sufficient to check it for $q(t)=t^k$. Then we have
$$
{\mathsf R}_{t^k, A} = -\frac{1}{8}    \sum_{s+r=k-1} {\mathsf R}_{A^s, A^r} 
$$
and, by using \eqref{eq:RicRAB},
$$
\begin{aligned}
\mathrm{Ric}& \left( {\mathsf R}_{t^k, A}  \right) = -\frac{1}{8}  \sum_{s+r=k-1} \bigl(\tr (A^s) A^r + \tr (A^r) A^s +  4 A^{k-1}\bigr) = \\
& -\frac{1}{4} \Big( \tr (A^0) A^{k-1} + \tr (A) A^{k-2} + \tr (A^2) A^{k-3} + \dots  + \tr (A^{k-1}) A^0\Big) -\frac{1}{2}  k A^{k-1}. \\ 
\end{aligned}
$$
Now taking into account the fact that each eigenvalue $\lambda_i$ of $A$ comes with multiplicity 2,  we can rewrite the above formula as
$$
\begin{aligned}
\mathrm{Ric} \left( {\mathsf R}_{t^k, A}  \right) &= -\frac{1}{2} \sum_{\lambda_i \in \mathrm{Spec}(A)} \Big( A^{k-1} 
+ \lambda_i A^{k-2} + \lambda_i^2 A^{k-3} + \dots  + \lambda_i^{k-1} \Id  \Big) -\frac{1}{2}  k A^{k-1} \\
&= -\frac{1}{2} \sum_{\lambda_i \in \mathrm{Spec}(A)}  q_{\lambda_i}(A) - \frac{1}{2}   q'(A), 
\end{aligned}
$$
where $q(t)=t^k$, as required.
\end{proof}
We have shown that an algebraic curvature operator ${\mathsf R}$ solving~\eqref{eq:secop} is of the form~\eqref{eq:polcurvop} 
up to adding a curvature operator $\upvarphi$ with image in the stabilizer $\mathfrak{g}_A$ of $A$ (of the (adjoint) action of $\mathfrak{u}(g,J)$ on hermitian endomorphisms). 
The form of this subalgebra is very much restricted by the Jordan normal form of $A$:

\begin{lem}
\label{lem:curvcommwithA}
Let $(\mathbb V,g,J)$ be a (pseudo-)hermitian vector space 
and let $A:\mathbb V\rightarrow \mathbb V$ be a hermitian endomorphism. Denote 
the complexification of $\mathbb{V}$ by $\mathbb{V}^\C=\mathbb{V}\otimes_\R \C$
and, for $\lambda\in \mathrm{Spec}(A)$, the generalized $\lambda$-eigenspace of $A$
by $\mathbb{V}_\lambda\subseteq \mathbb{V}^\C$. 
Suppose $\upvarphi\in \mathcal{K}(\mathbb V)$ satisfies
\begin{equation}\label{eq:commwithA0}
[\upvarphi(X),A]=0\quad\mbox{for all}\quad X\in \mathfrak{u}(g,J).
\end{equation}
Then, $\varphi(x,y,u,v)=0$ (we extend $\varphi$ complex-linearly to the complexification $\mathbb V^\C$)
whenever two vectors  are contained in different spaces $\mathbb{V}_\lambda$, $\mathbb{V}_{\lambda'}$,
$\lambda\neq \lambda'$.
\end{lem}

\begin{proof}
The proof follows immediately from \eqref{eq:commwithA0}, which equivalently reads $[\varphi(x,y),A]=0$ for all $x,y\in \mathbb V$,
in combination with the symmetries ($\mathcal{K}1$)--($\mathcal{K}4$) of a K\"ahler curvature tensor $\varphi(x,y,u,v)=g(x,\varphi(u,v)y)$ and the fact that generalized eigenspaces 
$\mathbb{V}_\lambda$, $\mathbb{V}_{\lambda'}$, $\lambda\neq \lambda'$, are perpendicular to each other.
\end{proof}

Of particular interest for us is the case $q(A)=0$ in the commutator relation~\eqref{eq:secop}. 
Part~(b) of the next proposition relates the statement of 
Theorem~\ref{thm:Bochner} for Bochner-flat metrics to the weakly Bochner-flat case from Theorem~\ref{thm:WBF}. 
Part~(c) is Theorem~\ref{thm:symmetricspace}.

\begin{prop}\label{prop:curvminpol}
Let $(\mathbb V,g,J)$ be a (pseudo-)hermitian vector space, 
let $A:\mathbb V\rightarrow \mathbb V$ be a hermitian endomorphism with characteristic polynomial $p(t)=\det_\C(t\cdot\Id-A)$
and let $\Theta(t)$ be a polynomial such that $\Theta(A)=0$. Then, we have the following:
\begin{enumerate}
\item[(a)] ${\mathsf R}={\mathsf R}_{\Theta,A}$ satisfies \eqref{eq:commwithA0}. In particular, $R_{\Theta,A}(x,y,u,v)=0$ 
whenever two vectors  are contained in different spaces $\mathbb{V}_\lambda$, $\mathbb{V}_{\lambda'}$,
$\lambda\neq \lambda'$ (in the notation of Lemma~\ref{lem:curvcommwithA}).

Conversely, if $[{\mathsf R}_{q,A}(X),A]=0$ for all $X\in \mathfrak{u}(g,J)$ and certain polynomial $q(t)$, 
we have $q(A)=0$ if we chose the zero-order term in $q(t)$ appropriately.
\item[(b)] Consider the polynomial $H(t)=\Theta'(t)+\Theta(t)\frac{p'(t)}{p(t)}$. Then,
$$
\Ric({\mathsf R}_{\Theta,A})=-\frac{1}{2}H(A).
$$

\item[(c)] ${\mathsf R}_{\Theta,A}$ satisfies the integrability condition \eqref{eq:condsymmetric}.
\end{enumerate}
\end{prop}

\begin{rem}\label{rem:BFtoWBF}
Since $\Theta(A)=0$ and therefore, the roots of $p(t)$ (i.e., the eigenvalues of $A$)
are also roots of $\Theta(t)$, we have that the function $H(t)$ from Proposition~\ref{prop:curvminpol}(b) is indeed a polynomial.
\end{rem}

\begin{proof}[Proof of Proposition \ref{prop:curvminpol}]
(a) From Lemma \ref{lem:curvpol}(b) we obtain $[{\mathsf R}_{\Theta,A}(X),A]=[X,\Theta(A)]=0$ for all $X\in \mathfrak{u}(g,J)$
as claimed. For the converse direction, we use again Lemma \ref{lem:curvpol}(b) and conclude 
$[X,q(A)]=[{\mathsf R}_{q,A}(X),A]=0$ for all $X\in \mathfrak{u}(g,J)$. Therefore $q(A)$ is proportional to $\Id$.

(b)  The formula for the Ricci tensor immediately follows from \eqref{eq:RicRTheta2} as in this case $\Theta (\lambda_i) = 0$, $p(t) =\prod (t-\lambda_i)$ and therefore 
$H(t) = \Theta'(t) + \sum_{\lambda_i\in\mathrm{Spec}(A)} \Theta_{\lambda_i}(t)$.

(c) Rewriting \eqref{eq:condsymmetric} as a condition on operators ${\mathsf R}_0:\mathfrak{u}(g,J)\rightarrow \mathfrak{u}(g,J)$,
it takes the form
$$
0=[{\mathsf R}_0(X),{\mathsf R}_0(Y)]-{\mathsf R}_0([{\mathsf R}_0(X),Y])\quad\mbox{for all }X,Y\in \mathfrak{u}(g,J).
$$
From \eqref{eq:RABopK} and \eqref{eq:polcurvop} we see that ${\mathsf R}_{\Theta,A}(X)$ is a linear combination
of terms of the form ${\mathsf R}_1(X)=A^rXA^s$ and ${\mathsf R}_2(X)=\tr(JA^rX)JA^s$. 
Thus, it suffices to verify the two equations
$$
{\mathsf R}_1\big([{\mathsf R}_{\Theta,A}(X),Y]\big)=[{\mathsf R}_{\Theta,A}(X),{\mathsf R}_1(Y)]
\quad\mbox{and}\quad
{\mathsf R}_2\big([{\mathsf R}_{\Theta,A}(X),Y]\big)=[{\mathsf R}_{\Theta,A}(X),{\mathsf R}_2(Y)].
$$
Using that by part (a) we have ${\mathsf R}_{\Theta,A}(X)A=A{\mathsf R}_{\Theta,A}(X)$, the first equation 
can be verified straight-forwardly. The second equation also follows easily from part (a): 
$$
{\mathsf R}_2
\big([{\mathsf R}_{\Theta,A}(X),Y]\big)=\tr\Big(JA^r\big({\mathsf R}_{\Theta,A}(X)Y-Y{\mathsf R}_{\Theta,A}(X)\big)\Big)JA^s
$$
$$
=\tr\big(JA^r{\mathsf R}_{\Theta,A}(X)Y\big)JA^s-\tr\big(JA^rY{\mathsf R}_{\Theta,A}(X)\big)JA^s=0
$$
which coincides with $[{\mathsf R}_{\Theta,A}(X), {\mathsf R}_2(Y)]=0$.
\end{proof}

\begin{rem}\label{rem:Bols1}
It is interesting to notice the following reformulation of the Bochner-flat condition.   
An algebraic curvature operator ${\mathsf R}: \mathfrak{u}(g,J) \to 
 \mathfrak{u}(g,J)$  is  Bochner-flat if and only if ${\mathsf R}$ can be written as ${\mathsf R}={\mathsf R}_{q,A}$ for some hermitian operator $A$ 
and  a quadratic polynomial $q$.  Indeed, according to the definition  \eqref{eq:Bochner} of the Bochner tensor,  the condition $\mathrm{Boch} = 0$ is equivalent to 
 $$
{\mathsf R} = \frac{1}{2(n+2)} {\mathsf R}_{\Id, \mathrm{Ric}} - \frac{\mathrm{Scal}}{8(n+1)(n+2)} {\mathsf R}_{\Id, \Id} = {\mathsf R}_{q, \mathrm{Ric}},
 $$
 where $q(t) = -\frac{2}{n+2} t^2  +     \frac{\mathrm{Scal}}{(n+1)(n+2)} t + \mathrm{const}$.  
Conversely, if ${\mathsf R} = {\mathsf R}_{q, A}$ with $q(t) = at^2 + bt + c$, then  $\mathsf R = -\frac{1}{8} (2a\,{\mathsf R}_{\Id, A} + b\, {\mathsf R}_{\Id,\Id})$, 
i.e., $\mathsf R$ is a linear combination of two Bochner-flat tensors.  
Note also that $\mathsf R$ is Bochner-flat if and only if $\mathsf R={\mathsf R}_{\Id, A}$ for a certain hermitian endomorphism $A$.
 \end{rem}

\subsection{The curvature in the weakly Bochner-flat case}
\label{subsec:Bochnerflatcond}

As in \S\ref{subsec:WBF} assume that a (pseudo-)K\"ahler structure $(g,J,\omega)$ is weakly Bochner-flat
with normalized Ricci tensor a linear combination of $\Id$ and a hermitian solution $A$ of \eqref{eq:main}. 
Then, $\Ric$ is given by~\eqref{eq:WBF} for certain constants $c_2,c_1$ which appear as
coefficients of highest order in the degree $\leq \ell+1$ polynomial $H(t)=\Theta_i'(t) - \Theta_i (t)\frac{p'_\c(t)}{p_\c(t)}$ from Theorem~\ref{thm:WBF}. 
Taking the trace from \eqref{eq:WBF} and replacing $c_1,c_2$ by constants $C_1,C_2$ defined by 
\begin{equation}\label{eq:constants}
c_2=(n+2)C_2\mbox{ and }c_1=\frac{C_2}{2}\tr(A_\c)+(n+1)C_1,
\end{equation}
we obtain
$$
\frac{\Scal}{n+1}=-(n+2)C_2\mu_1-C_2\tr(A_\c)-nC_1.
$$
Inserting these formulas for $\Ric$, $\Scal$, $c_1$ and $c_2$ 
into the formula for $R$ from \eqref{eq:Bochner} and using the notation from 
\S\ref{subsec:algprelim2}, we see that 
\begin{equation}\label{eq:curvWBF}
R=\mathrm{Boch}-\tfrac{1}{8}\big((C_2\mu_1+C_1)R_{\Id,\Id}+2C_2 R_{\Id,A}\big)=R_{q,A}+\mathrm{Boch},
\end{equation}
for a polynomial 
$$
q(t)=C_2t^2+(C_2\mu_1+C_1)t+f_0,
$$
where $f_0$ can be any function on $M$ (it does not contribute to $R_{q,A}$). 
If the (pseudo-)K\"ahler structure is actually Bochner-flat, we obtain
\begin{equation}\label{eq:curvBF}
R=R_{q,A}=-\tfrac{1}{8}\big((C_2\mu_1+C_1)R_{\Id,\Id}+2C_2 R_{\Id,A}\big).
\end{equation}
Notice that the condition~\eqref{eq:curvBF} is necessary and sufficient for the (pseudo-)K\"ahler structure $(g,J,\omega)$ 
to be Bochner-flat  (cf.  Remark~\ref{rem:Bols1}).   
Thus,  our goal is to show that \eqref{eq:curvBF} is equivalent to conditions (a), (b) and (c) of Theorem \ref{thm:Bochner}.

\begin{rem}
\label{rem:Bols2}
If the conditions (a) and (b) of Theorem \ref{thm:Bochner} are fulfilled, then both $\Theta(t)$ and $H(t)= \Theta'(t) + \Theta (t) \frac{p'_{\c}(t)}{p_{\c}(t)}$ 
are polynomials  of degree $\ell+2$ and $\ell+1$ respectively.  The leading coefficients of these  polynomials 
$\Theta (t) = C_2 t^{\ell +2} + C_1 t^{\ell +1} + C_0 t^{\ell} + \dots$ and $H (t) = c_2 t^{\ell +1} + c_1 t^{\ell} + c_0 t^{\ell-1} + \dots$ 
are related as follows  (this can be easily checked by a straight-forward computation):
\begin{equation}
\label{eq:relCc}
\begin{aligned}
c_2 &= (n+2) C_2 \\
c_1 &= (n+1) C_1 +  \frac{C_2}{2} \tr (A_\c) \\
c_0 &= n C_0 +  \frac{C_1}{2} \tr  (A_\c) + \frac{C_2}{2} \tr  (A^2_\c) 
\end{aligned}
\end{equation}
In particular, the notation for these coefficients agrees with  \eqref{eq:constants}.  
If the constant block $(g_\c,\omega_\c,A_\c)$ is absent, we assume that 
$\tr  (A_\c) = \tr  (A^2_\c) = 0$, i.e., the corresponding terms in \eqref{eq:relCc} simply disappear.
\end{rem}

%%%%%%%%%%%%%%%%%%%%%%%%%%%%%%%%%%%%%

\subsection{Conditions on the $\Theta_i$'s}
\label{subsec:BochnerflatTheta}

We start by recalling a result from the general theory of c-projectively equivalent metrics, cf.\
\cite[Proposition 4]{ApostolovI}, \cite[Lemmas 2.1 and A.1]{BMR} and \cite[Equation (1.3)]{Mikes}:

\begin{lem}\label{lem:prolongation}
Let $(M,g,J)$ be a (pseudo-)K\"ahler manifold of real dimension $2n\geq 4$
and let $A$ be a hermitian solution of \eqref{eq:main} and $\Lambda=\tfrac{1}{4}\gr\,\tr(A)$. 
\begin{enumerate}
\item[(a)] The endomorphism $\nabla\Lambda$ is hermitian and the curvature operator ${\mathsf R}:\mathfrak{u}(g,J)\rightarrow \mathfrak{u}(g,J)$ 
satisfies
\begin{equation}\label{eq:secopcpro}
[{\mathsf R}(X),A]=4[X,\nabla\Lambda]\quad\mbox{for all}\quad X\in \mathfrak{u}(g,J).
\end{equation}
\item[(b)] The endomorphism $\nabla\Lambda$ satisfies the relation
\begin{equation}\label{eq:prolong}
\nabla\Lambda=\tfrac{1}{2n}\big( \tr (\nabla \Lambda)\, \Id+\tfrac{1}{2}J \,{\mathsf R}(JA)-A\,\Ric\big),
\end{equation}       
where ${\mathsf R}(JA)$ is the action of ${\mathsf R}$ on the skew-hermitian endomorphism $JA\in\mathfrak{u}(g,J)$ (cf.\ \S\ref{subsec:algprelim2}).
\end{enumerate}
\end{lem}
Lemma~\ref{lem:prolongation}(a) provides the link to the theory described in \S\ref{subsec:algprelim2}:
Equation~\eqref{eq:secopcpro} takes the same form as Equation~\eqref{eq:secop0} from Lemma~\ref{lem:curvpol}.
In particular, we conclude that at each point of $M$ we have 
$\nabla\Lambda=\tfrac{1}{4}q(A)$ for some polynomial $q(t)=\sum_{k=0}^N a_k t^k$ 
with coefficients $a_k$ which may depend on the points of $M$.  The  next lemma  describes those (pseudo-)K\"ahler structures  for which this polynomial is quadratic.  

\begin{lem}\label{lem:Bols1}
Let $(g,J,\omega)$ be the (pseudo-)K\"ahler structure from \eqref{eq:gjwrhocoord} and let $A$ 
be the solution of \eqref{eq:main} given by \eqref{eq:Arhocoord}.   
\begin{enumerate}
\item[(a)] Assume that the parameters $\Theta_i$ from~\eqref{eq:gjwrhocoord} satisfy 
conditions {\rm (a)} and {\rm (b)} of Theorem~\ref{thm:Bochner}. Then $\nabla\Lambda=\tfrac{1}{4}q(A)$ with $q(t)$ being a quadratic polynomial of the form
\begin{equation}
\label{eq:polynomialBols}
q(t)=C_2 t^2+(C_2\mu_1+C_1)t+C_2 (\mu_1^2-\mu_2)+C_1\mu_1+C_0,
\end{equation}
where $C_2$, $C_1$ and $C_0$ are the leading coefficients of the polynomial $\Theta (t) = \Theta_i(t)$. 

\item[(b)] Conversely,  if $\nabla\Lambda=\tfrac{1}{4}q(A)$ with $q(t)$ given by~\eqref{eq:polynomialBols} for some constants $C_0$, $C_1$ and $C_2$, 
then the $\Theta_i$  satisfy conditions {\rm (a)} and {\rm (b)} of Theorem~\ref{thm:Bochner} and $\Theta (t) = \Theta_i(t)$ has leading coefficients $C_2$, $C_1$ and $C_0$.
\end{enumerate}
\end{lem}

\begin{proof}
By using the formulas~\eqref{eq:gjwrhocoord} for $(g,J,\omega)$, we can compute $\nabla\Lambda$ explicitly:

\begin{lem}\label{lem:nablalambda}
Let $(g,J,\omega)$ be the (pseudo-)K\"ahler structure from \eqref{eq:gjwrhocoord}
and let $A$ be the solution of \eqref{eq:main} given by \eqref{eq:Arhocoord}.
\begin{enumerate}
\item[(a)] We have 
$$
\nabla\Lambda(\gr\,\rho_i)=\frac{1}{4}\left(\frac{\D}{\D \rho_i} \sum_{j=1}^\ell \frac{\Theta_j(\rho_j)}{\Delta_j}\right)\gr\,\rho_i
$$
for $i=1,\dots,\ell$.
\item[(b)] We have 
$$
\nabla\Lambda|_{\mathcal{F}^\perp}=\frac{1}{4}\sum_{i=1}^\ell\frac{\Theta_i(\rho_i)}{\Delta_i}(\rho_i\Id-A_\c)^{-1},
$$
where $\mathcal F=\mathrm{span}\{\gr\,\rho_i,J\gr\,\rho_i:i=1,\dots,\ell\}$ (such that the orthogonal complement $\mathcal{F}^\perp$
is the direct sum of the generalized eigenspaces of $A$ corresponding to the constant eigenvalues).
\item[(c)] We have
$$
\tr  ( \nabla \Lambda )= \frac{1}{2} \sum_{i=1}^\ell\frac{H_i(\rho_i)}{\Delta_i}, 
$$
where $H_i(t)=\Theta_i'(t)+\Theta_i(t)\tfrac{p_\c'(t)}{p_\c(t)}$ for all $i=1,\dots,\ell$.
In particular,  if $H_i(t)= H(t) = c_2 t^{\ell + 1} + c_1 t^{\ell} + c_0 t^{\ell - 1} +\dots$  is a polynomial of degree $\leq\ell+1$, then
\begin{equation}
\label{eq:trDL}
\tr  ( \nabla \Lambda )  = \frac{1}{2} \left(c_2 \left(\sum_{i=1}^\ell \rho_i^2 + \mu_2\right) + c_1\mu_1 + c_0\right).
\end{equation}
\end{enumerate}
\end{lem}

\begin{proof}
Define $m_i$ via $\nabla\Lambda(\gr\,\rho_i)=m_i\gr\,\rho_i$.
Using $\Lambda=\tfrac{1}{2}\sum_{i=1}^\ell\gr\,\rho_i$ we calculate for an arbitrary tangent vector $u$ that
$$
u \big(g(\Lambda,\Lambda)\big)=2g(\nabla_u \Lambda,\Lambda)=2g(\nabla_\Lambda \Lambda,u)=\sum_{i=1}^\ell g\big(\nabla\Lambda(\gr\,\rho_i),u\big)
= \sum_{i=1}^\ell m_i g(\gr\,\rho_i,u).
$$
Hence, $\d g(\Lambda,\Lambda)=\sum_{i=1}^\ell m_i \d\rho_i$ or $m_i=\tfrac{\D}{\D \rho_i} g(\Lambda,\Lambda)$. 
By \eqref{eq:gjwrhocoord} we have 
$$
g(\Lambda,\Lambda)=\frac{1}{4}\sum_{i=1}^\ell g(\gr\,\rho_i,\gr\,\rho_i)=\frac{1}{4}\sum_{i=1}^\ell \frac{\Theta_i(\rho_i)}{\Delta_i}
$$
from which (a) follows. To prove (b), let $u\in \mathcal{F}^\perp$. Using \eqref{eq:main}, we compute for each $i\in \{1,\dots,\ell\}$
$$
0=\nabla_u\big((A-\rho_i\Id)\gr\,\rho_i\big)=\frac{1}{2}\frac{\Theta_i(\rho_i)}{\Delta_i} u+(A-\rho_i\Id)\nabla_u\gr\,\rho_i.
$$
Consequently, $\nabla_u\Lambda=-\frac{1}{4}\sum_{i=1}^\ell\frac{\Theta_i(\rho_i)}{\Delta_i}(A_\c-\rho_i\Id)^{-1} u$
from which (b) follows.

To prove (c),   we notice that (a) gives us all the eigenvalues of $\nabla \Lambda$ related to the nonconstant block. 
Taking into account both constant and nonconstant blocks,  formula \eqref{eq:frombols1} and the fact that each eigenvalue comes with multiplicity two, we get 
$$
\tr (\nabla\Lambda) = \frac{1}{2} \sum_{i=1}^\ell\left(\frac{\D}{\D \rho_i} \sum_{j=1}^\ell \frac{\Theta_j(\rho_j)}{\Delta_j}\right) 
+ \frac{1}{4}\tr \left(\sum_{i=1}^\ell\frac{\Theta_i(\rho_i)}{\Delta_i}(\rho_i\Id-A_\c)^{-1}\right) =
$$
$$
 =\frac{1}{2}  \sum_{i=1}^\ell \frac{\Theta'_i(\rho_i)}{\Delta_i} + \frac{1}{4} \sum_{i=1}^\ell \left(\frac{\Theta_i(\rho_i)}{\Delta_i}\tr (\rho_i\Id-A_\c)^{-1}\right) = 
$$
$$
=\frac{1}{2} \sum_{i=1}^\ell \frac{1}{\Delta_i }\left(\Theta'_i(\rho_i) 
+ \sum_{\lambda_\alpha \in \mathrm{Spec}{A_\c}} \frac{\Theta_i(\rho_i)}{\rho_i - \lambda_\alpha} \right) =\frac{1}{2} \sum_{i=1}^\ell \frac{H_i(\rho_i)}{\Delta_i }.
$$
The final formula for $H_j(t)=H(t)= c_2 t^{\ell + 1} + c_1 t^{\ell} + c_0 t^{\ell - 1} +\dots$ follows from the Vandermonde identities~\eqref{eq:VandId}.
\end{proof}
Let us continue with the proof of Lemma~\ref{lem:Bols1}: let $q(t)$ be the polynomial given by~\eqref{eq:polynomialBols} 
for certain constants $C_0$, $C_1$ and $C_2$. Then, $\nabla\Lambda=\tfrac{1}{4}q(A)$ holds if and only if
\begin{equation}\label{eq:system}
\nabla\Lambda(\gr\,\rho_i)=\tfrac{1}{4}q(\rho_i)\gr\,\rho_i\quad\forall i=1,\dots,\ell
\quad\mbox{and}\quad
\nabla\Lambda|_{\mathcal{F}^\perp}=\tfrac{1}{4}q(A_\c).
\end{equation}
By Lemma~\ref{lem:nablalambda}(a), the first set of equations in~\eqref{eq:system} is satisfied if and only if
\begin{equation}\label{eq:determineTheta}
\frac{\D}{\D \rho_i} \sum_{j=1}^\ell \frac{\Theta_j(\rho_j)}{\Delta_j}
=C_2 \rho_i^2+(C_2\mu_1+C_1)\rho_i+C_2 (\mu_1^2-\mu_2)+C_1\mu_1+C_0
\end{equation}
for each $i=1,\dots,\ell$. Lemma~\ref{lem:VandId}(b) implies that functions $\Theta_i(t)$ solving \eqref{eq:determineTheta}
must be equal to a polynomial $\Theta_i(t)=\Theta(t)$ of degree $\leq \ell+2$ (independent of $i$). 
Moreover, summation over all $i=1,\dots,\ell$ in~\eqref{eq:determineTheta} together with the identity~\eqref{eq:frombols1} yields
$$
\sum_{i=1}^\ell \frac{\Theta'(\rho_i)}{\Delta_i}
=(\ell+2)C_2(\mu_1^2-\mu_2)+(\ell+1)C_1\mu_1+\ell C_0.
$$
Comparing with \eqref{eq:VandId}, we see that $C_2,C_1,C_0$ are the leading coefficients of the polynomial $\Theta(t)$.
Conversely, the choice $\Theta_i(t)=\Theta(t)=C_2t^{\ell+2}+C_1t^{\ell+1}+C_0t^{\ell}+\dots$ for polynomials $\Theta_i(t)$ solves~\eqref{eq:determineTheta}
as can be verified straight-forwardly by using the Vandermonde identities~\eqref{eq:VandId}.

\medskip
For proving Lemma~\ref{lem:Bols1}, it remains to show that under the assumption that $\Theta_i(t)=\Theta(t)=C_2t^{\ell+2}+C_1t^{\ell+1}+C_0t^\ell+\dots$,
we have $\nabla\Lambda|_{\mathcal{F}^\perp}=\tfrac{1}{4}q(A_\c)$ satisfied if and only if $\Theta(A_\c)=0$. 
By Lemma~\ref{lem:nablalambda}(b), $\nabla\Lambda|_{\mathcal{F}^\perp}=\tfrac{1}{4}q(A_\c)$ is satisfied if and only if
\begin{equation}\label{eq:system2}
\sum_{i=1}^\ell\frac{\Theta_i(\rho_i)}{\Delta_i}(\rho_i\Id-A_\c)^{-1}=q(A_\c).
\end{equation}
By Lemma~\ref{lem:VandId}(c), we have 
\begin{equation}\label{eq:deriv}
\sum_{i=1}^\ell\frac{\Theta(\rho_i)}{\Delta_i(\rho_i-t)}=-\frac{\Theta(t)}{p_\nc(t)}+q(t)
\end{equation}
such that the left-hand side of \eqref{eq:system2} is the same as $-\Theta(A_\c)p_\nc(A_\c)^{-1}+q(A_\c)$. Thus, \eqref{eq:system2} is equivalent 
to $\Theta(A_\c)=0$ as we claimed.   
\end{proof}
Now, if we assume that $(g,J,\omega)$ is weakly Bochner-flat, then we can obtain the formula for $q(t)$
by inserting the formulas \eqref{eq:WBF} and \eqref{eq:curvWBF} for $\Ric$ and $R$ 
into the right-hand side of the formula \eqref{eq:prolong} for $\nabla\Lambda$, cf.\
\cite[\S 2.3]{ApostolovI} and \cite[Equation (2.14)]{Bryant}:

\begin{lem}
\label{lem:polynomial}
Let $(g,J,\omega)$ be the (pseudo-)K\"ahler structure from \eqref{eq:gjwrhocoord}
and let $A$ be the solution of \eqref{eq:main} given by \eqref{eq:Arhocoord}. Suppose $(g,J,\omega)$ is weakly-Bochner flat
with normalized Ricci tensor \eqref{eq:normRic} a linear combination of $A$ and $\Id$.
Let $c_2,c_1,c_0$ be the leading coefficients of the degree $\leq \ell+1$ polynomial $H(t)=\Theta_i'(t) + \Theta_i(t) \frac{p'_\c(t)}{p_\c(t)}$
such that $\Ric$ is given by \eqref{eq:WBF}. Then we have 
\begin{equation}\label{eq:Pol+Boch}
\nabla\Lambda=\tfrac{1}{4n}J\mathsf{Boch}(JA)+\tfrac{1}{4} q(A),
\end{equation}
where $q(t)$ is a quadratic polynomial given by \eqref{eq:polynomialBols}  with constants $C_2,C_1, C_0$ related to $c_2, c_1, c_0$ by \eqref{eq:relCc}.   

If moreover, $\nabla\Lambda=\tfrac{1}{4}q(A)$ for $q(t)$ given by \eqref{eq:polynomialBols} with some constants $C_2, C_1, C_0$, then 
these constants are related to $c_2, c_1, c_0$ as in \eqref{eq:relCc} and, therefore, $\mathsf{Boch}(JA)=0$.
\end{lem}

\begin{proof}
The curvature $R$ is given by the formula \eqref{eq:curvWBF}. Using \eqref{eq:RABopK}
in
${\mathsf R}(JA)=\mathsf{Boch}(JA)-\tfrac{1}{8}\big((C_2\mu_1+C_1){\mathsf R}_{\Id,\Id}(JA)+2C_2 {\mathsf R}_{\Id,A}(JA)\big)$,
we compute
$$
\tfrac{1}{2}J{\mathsf R}(JA)=\tfrac{1}{2}J\mathsf{Boch}(JA)-C_2 A^2-\tfrac{1}{2}\big(2C_2\mu_1+C_1+\tfrac{C_2}{2}\tr(A_\c)\big)A
-\tfrac{1}{4}\big((C_2\mu_1+C_1)\tr(A)+C_2\tr(A^2)\big)\Id.
$$
On the other hand, by \eqref{eq:WBF} and \eqref{eq:relCc},
$$
A\,\Ric =-\tfrac{1}{2}(n+2)C_2 A^2-\tfrac{1}{2}\big((n+2)C_2\mu_1+\tfrac{C_2}{2}\tr(A_\c)+(n+1)C_1\big)A
$$
so that
$$
\tfrac{1}{2}J{\mathsf R}(JA)-A\,\Ric 
=\tfrac{1}{2}J\mathsf{Boch}(JA)+\tfrac{n}{2}C_2 A^2+\tfrac{n}{2}(C_2\mu_1+C_1)A
-\tfrac{1}{4}\big((C_2\mu_1+C_1)\tr(A)+C_2\tr(A^2)\big)\Id.
$$
It remains to substitute this expression in \eqref{eq:prolong} and use formula \eqref{eq:trDL} for $\tr(\nabla\Lambda)$ to get:
$$
\nabla\Lambda=\tfrac{1}{2n}\big(\tfrac{1}{2}J{\mathsf R}(JA)-A\,\Ric\big)+\tfrac{1}{2n}\tr(\nabla\Lambda) \Id
$$
$$
=\tfrac{1}{4n}J\mathsf{Boch}(JA)+\tfrac{1}{4}\bigg( C_2 A^2+(C_2\mu_1+C_1)A-\tfrac{1}{2n}\big((C_2\mu_1+C_1)\tr(A)+C_2\tr(A^2)\big)   \Id +
$$
$$
+ \frac{1}{n} \big(  c_2 (\sum\rho_i^2 + \mu_2) + c_1\mu_1 + c_0\big) \Id
\bigg) =\tfrac{1}{4n}J\mathsf{Boch}(JA)+\tfrac{1}{4} q(A)
$$
as claimed.  The final part follows from Lemma~\ref{lem:Bols1}(b)  and Remark~\ref{rem:Bols2}.
\end{proof}

%%%%%%%%%%%%%%%%%%%%%%%%%%%%%%%%%%
We now use Lemmas~\ref{lem:Bols1} and~\ref{lem:polynomial} to clarify the geometric meaning of conditions (a) and (b) of Theorem~\ref{thm:Bochner}. 
For the (pseudo-)K\"ahler structure $(g,J,\omega)$ from \eqref{eq:gjwrhocoord} we say that
a tangent vector is \emph{vertical} if it is contained in the vertical distribution
$\mathcal F=\mathrm{span}\{\gr\,\rho_i,J\gr\,\rho_i:i=1,\dots,\ell\}$. It is called \emph{horizontal} if it is contained
in the orthogonal complement $\mathcal{F}^\perp$. Note that $\mathcal F$ is the direct sum of all eigenspaces of $A$ from  \eqref{eq:Arhocoord}
corresponding to the nonconstant eigenvalues $\rho_1,\dots,\rho_\ell$. Thus, $\mathcal{F}^\perp$ 
is the direct sum of the generalized eigenspaces of $A$ corresponding to the constant eigenvalues.
We will call a tensor \emph{horizontal} if it vanishes when a vertical tangent vector is inserted. 

\begin{prop}\label{prop:Bochner1}
Consider the (pseudo-)K\"ahler structure $(g,J,\omega)$ from \eqref{eq:gjwrhocoord} and let
$A$ be given by \eqref{eq:Arhocoord}.
If $(g,J,\omega)$ is Bochner-flat with normalized Ricci tensor \eqref{eq:normRic}
a constant linear combination of $g(A\cdot,\cdot)$ and $g$, then conditions (a) and (b) of Theorem \ref{thm:Bochner} hold, i.e.,
\begin{enumerate}
\item[(a)] $\Theta_i(t)=\Theta(t)$ is a polynomial of degree $\leq \ell+2$ independent of $i$ and
\item[(b)] $\Theta(A_\c)=0$.
\end{enumerate} 
Conversely, suppose $(g,J,\omega)$ is weakly Bochner-flat with normalized Ricci tensor \eqref{eq:normRic}
a linear combination of $g(A\cdot,\cdot)$ and $g$ and it satisfies conditions {\rm (a)} and {\rm (b)}
from above. Then, the Bochner tensor is horizontal.
\end{prop}

%%%%%%%%%%%%%%%%%%%%%%%%%%%%%%%%

\begin{proof}
Assume that $(g,J,\omega)$ is Bochner-flat with normalized Ricci tensor a linear combination of $A$ and $\Id$. 
Then by Lemma~\ref{lem:polynomial},   $\nabla\Lambda = \frac{1}{4} q(A)$  
with $q(t)$ given by~\eqref{eq:polynomialBols} and therefore conditions (a) and (b)  follow from Lemma~\ref{lem:Bols1}(b).

Conversely, suppose that $(g,J,\omega)$ is weakly Bochner-flat with normalized Ricci tensor
a linear combination of $g(A\cdot,\cdot)$ and $g$ and that conditions {\rm (a)} and {\rm (b)} are satisfied.   
Let $\Theta(t)=C_2 t^{\ell+2}+C_1t^{\ell+1}+C_0t^\ell+\dots$. By Lemma \ref{lem:Bols1} we have $\nabla\Lambda=\tfrac{1}{4}q(A)$,  
where $q(t)$ is of the form \eqref{eq:polynomialBols}.  By Lemma \ref{lem:curvpol}(b) and Lemma \ref{lem:prolongation}(a), the curvature tensor $R$
of $g$ is given by $R=R_{q,A}+\varphi$, where the algebraic curvature tensor $\varphi$ satisfies
\begin{equation}\label{eq:commwithA1}
[\varphi(u,v),A]=0\quad\mbox{for all tangent vectors }u,v.
\end{equation}
On the other hand, the weakly Bochner-flat condition implies that the curvature tensor $R$ is given by \eqref{eq:curvWBF}, i.e., 
$$
R=R_{q,A}+\mathrm{Boch}.
$$ 
\weg{with $q(t)$ given by \eqref{eq:polynomialBols}  (notice that a priori the polynomials $q(t)$ in formulas 
\eqref{eq:polynomialBols} and \eqref{eq:curvWBF}  might be different).} Hence we conclude that $\varphi = \mathrm{Boch}$ and 
\begin{equation}\label{eq:commwithA2}
[\mathrm{Boch}(u,v),A]=0\quad\mbox{for all tangent vectors }u,v.
\end{equation} 
By Lemma \ref{lem:curvcommwithA}, in order to prove that $\mathrm{Boch}$ is horizontal, 
it suffices to show that 
\begin{equation}\label{eq:varphihorizontal}
\mathrm{Boch}\big(\gr\,\rho_i,J\gr\,\rho_i\big)\gr\,\rho_i=0\mbox{ for all }i=1,\dots,\ell.
\end{equation}
By Lemma~\ref{lem:polynomial}, $\mathrm{Boch}$ also satisfies $\mathsf{Boch}(JA)=0$.  
We see from \eqref{eq:hamrhocoord} (reformulated using the notation from \S \ref{subsec:algprelim2})
that
$$
g(JA\cdot,\cdot)=\frac{1}{2}\sum_{i=1}^\ell\rho_i\frac{\Delta_i}{\Theta_i}g(\gr\,\rho_i\wedge_J J\gr\,\rho_i\cdot,\cdot)+\omega_\c(p_\nc(A_\c)A_\c\cdot,\cdot)
$$
and therefore $JA=\tfrac{1}{2}\sum_{i=1}^\ell\rho_i\frac{\Delta_i}{\Theta_i}\gr\,\rho_i\wedge_J J\gr\,\rho_i+E$, where $E:\mathcal{F}^\perp\rightarrow \mathcal{F}^\perp$ is some endomorphism
of $\mathcal{F}^\perp$ viewed as an endomorphism of the tangent bundle $\mathcal{F}\oplus\mathcal{F}^\perp$ by setting it zero on $\mathcal{F}$.
By Lemma \ref{lem:curvcommwithA}, we have 
$$
\mathrm{Boch}(\gr\,\rho_j, J\gr\,\rho_j)\gr\,\rho_i=0\quad\forall i\neq j\quad\mbox{and}\quad\mathsf{Boch}(E)\gr\,\rho_i=0\quad\forall i
$$
such that 
$$
0=\mathsf{Boch}(JA)\gr\,\rho_i=2\rho_i\frac{\Delta_i}{\Theta_i}\mathrm{Boch}\left(\gr\,\rho_i,J\gr\,\rho_i\right)\gr\,\rho_i
$$
for each $i=1,\dots,\ell$. We have shown that~\eqref{eq:varphihorizontal} is satisfied and this completes the proof of the proposition. 
\end{proof}

\subsection{Conditions on $(g_\c,\omega_\c)$}
\label{subsec:Bochnerflatconst}

The proof of the next statement will complete the proof of Theorem \ref{thm:Bochner}:

\begin{prop}\label{prop:Bochner2}
Consider the (pseudo-)K\"ahler structure $(g,J,\omega)$ from \eqref{eq:gjwrhocoord} and let
$A$ be given by \eqref{eq:Arhocoord}. Suppose $(g,J,\omega)$ is weakly Bochner-flat with normalized Ricci tensor \eqref{eq:normRic}
a linear combination of $g(A\cdot,\cdot)$ and $g$ and that conditions (a) and (b) of Proposition \ref{prop:Bochner1}
are satisfied. Then, $(g,J,\omega)$ is Bochner-flat if and only if the curvature tensor $R_\c$ of $(g_\c,\omega_\c)$
satisfies $R_\c=R_{g_\c,\Theta,A_\c}$.
\end{prop}

\begin{proof}
By \eqref{eq:curvBF} and Proposition~\ref{prop:Bochner1}, $(g,J,\omega)$ is Bochner-flat if and only if 
$$
R(x,y,u,v)=R_{g,q,A}(x,y,u,v)
$$
for all horizontal vectors $x,y,u,v$. We write $R_H(x,y,u,v)=R(x,y,u,v)$, that is, $R_H$ denotes the horizontal part
of the curvature tensor $R$. Moreover, we obviously have $R_{q,A}(x,y,u,v)=R_{g_H,q,A_\c}(x,y,u,v)$, where $g_H=g_\c(p_\nc(A_\c)\cdot,\cdot)$
denotes the horizontal part of the metric $g$. Thus, when written down more compactly, we have that $(g,J,\omega)$ is Bochner-flat if and only if 
\begin{equation}\label{eq:horizBFcond0}
R_H=R_{g_H,q,A_\c}.
\end{equation}
Let us relate $R_H$ to the curvature tensor $R_\c$ of the metric $g_\c$: using the formulas from \cite[Proposition 9]{ApostolovI} resp.\
the second part of the proof of \cite[Proposition 17]{ApostolovI} (the derivation remains the same in arbitrary signature), 
one calculates that 
\begin{equation}\label{eq:horizcurv}
\begin{array}{c}
\displaystyle R_H(x,y,u,v)=g_\c(p_\nc(A_\c)x,R_\c(u,v)y)
\vspace{1mm}\\
\displaystyle -g\big(C(u,x),C(v,y)\big)+g\big(C(v,x),C(u,y)\big)
+g\big(C(y,x),C(u,v)\big)-g\big(C(y,x),C(v,u)\big),
\end{array}
\end{equation}
for all horizontal vectors $x,y,u,v$, where
$$
C(x,y)=\frac{1}{2}\sum_{r=1}^\ell (-1)^r\big(\omega_\c(A_\c^{\ell-r}x,y)K_r+\omega_\c(A_\c^{\ell-r}Jx,y)JK_r\big)
$$
and $K_i=J\gr\,\mu_i$ are the coordinate vector fields $\tfrac{\D}{\D t_i}$ from~\eqref{eq:gjwrhocoord}. 
Note that \eqref{eq:horizcurv} holds for the (pseudo-)K\"ahler structure from~\eqref{eq:gjwrhocoord}
-- we do not have to assume additional conditions such as Bochner-flatness. 
A straight-forward calculation using~\eqref{eq:gjwrhocoord} yields 
\begin{equation}\label{eq:GrayONeillTensor}
\begin{array}{c}
\displaystyle g\big(C(x,y),C(u,v)\big)=
\vspace{1mm}\\
\displaystyle  =\frac{1}{4}\sum_{i=1}^\ell\frac{\Theta(\rho_i)}{\Delta_i}\Big(
g_\c\big(p_\nc(A_\c)(A_\c-\rho_i\Id)^{-1}x,y\big) \, g_\c\big(p_\nc(A_\c)(A_\c-\rho_i\Id)^{-1}u,v\big)
\vspace{1mm}\\
\displaystyle + \, \omega_\c\big(p_\nc(A_\c)(A_\c-\rho_i\Id)^{-1}x,y\big) \, \omega_\c\big(p_\nc(A_\c)(A_\c-\rho_i\Id)^{-1}u,v\big)\Big).
\end{array}
\end{equation}
Note that by assumption we have $\Theta(t)=C_2t^{\ell+2}+C_1 t^{\ell+1}+C_0t^\ell+\dots$ for certain constants $C_0$, $C_1$ and $C_2$.
For further simplification of Equation~\eqref{eq:horizcurv},
we use the notation introduced in \S\ref{subsec:algprelim2}: 
the hermitian vector space $\mathbb V$ w.r.t.\ which the K\"ahler curvature tensors from \S\ref{subsec:algprelim2} 
are considered is a tangent space $T_p S$ of $S$ with complex structure $J_\c$. 
Note that $(T_p S,J_\c)$ is canonically isomorphic to $(\mathcal{F}_q^\perp,J)$ for any point $q$ in the fiber over $p$
(cf.\ Remark~\ref{rem:naturalcoord}). For the hermitian metric on $\mathbb V$ we use $g_H=g_\c(p_\nc(A_\c)\cdot,\cdot)$. 
Combining~\eqref{eq:horizcurv} with~\eqref{eq:GrayONeillTensor}
and the definition~\eqref{eq:RAB} of curvature tensors of the form $R_{g,A,B}$, we obtain
\begin{equation}\label{eq:step0}
R_H=R_{g_H}
-\frac{1}{8}\sum_{i=1}^\ell\frac{\Theta(\rho_i)}{\Delta_i}R_{g_H,(A_\c-\rho_i\Id)^{-1},(A_\c-\rho_i\Id)^{-1}},
\end{equation}
where we defined $R_{g_H}(x,y,u,v)=g_H(x,R_\c(u,v)y)=g_\c(p_\nc(A_\c)x,R_\c(u,v)y)$ for all horizontal vectors $x,y,u,v$.
Since the metrics $g_\c$ and $g_H$ on $S$ have the same Levi-Civita connections (since $A_\c$ is $g_\c$-parallel by Theorem~\ref{thm:localclass}), 
the $(1,3)$-curvature tensors of $g_H$ and $g_\c$ coincide. Thus, $R_{g_H}$ is the $(0,4)$-curvature tensor of $g_H$.
The computations from Example~\ref{exmp:main} show that Equation~\eqref{eq:step0} can be written as
$$
R_H=R_{g_H}+R_{g_H,f,A_\c},
$$
for the function $f(t)=\sum_{i=1}^\ell\frac{\Theta(\rho_i)}{\Delta_i(\rho_i-t)}$.
Moreover, by Lemma~\ref{lem:VandId}(c), we have $f(t)=-\tfrac{\Theta(t)}{p_\nc(t)}+q(t)$
for $q(t)=C_2t^2+(C_2\mu_1+C_1)t+\dots$ such that we obtain
\begin{equation}\label{eq:step}
R_H=R_{g_H}-R_{g_H,\Theta/p_\nc,A_\c}+R_{g_H,q,A_\c}.
\end{equation}
From \eqref{eq:step} we see that the horizontal Bochner-flat condition~\eqref{eq:horizBFcond0} 
is equivalent to
\begin{equation}\label{eq:horizBFcond}
R_{g_H}=R_{g_H,\Theta/p_\nc,A_\c}.
\end{equation}
We claim that \eqref{eq:horizBFcond} is actually equivalent to
$$
R_\c=R_{g_\c,\Theta,A_\c}.
$$
For proving this, let us evaluate the right-hand side of \eqref{eq:horizBFcond} using $\Theta(A_\c)=0$:
for $X\in \mathfrak{so}(g_H)$ we compute for the Riemannian curvature operator 
$\widetilde{\mathsf{R}}_{g_H,\Theta/p_\nc,A_\c}: \mathfrak{so}(g_H)\rightarrow  \mathfrak{so}(g_H)$
\begin{equation}\label{eq:firsteq}
\begin{array}{c}
\displaystyle \widetilde{\mathsf{R}}_{g_H,\Theta/p_\nc,A_\c}(X)=\frac{\d}{\d t}\Big|_{t=0}\Theta(A_\c+tX)p_\nc(A_\c+tX)^{-1}
\vspace{1mm}\\
\displaystyle =\frac{1}{2}(\widetilde{\mathsf{R}}_{g_H,\Theta,A_\c}(X)p_\nc(A_\c)^{-1}+p_\nc(A_\c)^{-1}\widetilde{\mathsf{R}}_{g_H,\Theta,A_\c}(X))
=\frac{1}{2}\widetilde{\mathsf{R}}_{g_H,\Theta,A_\c}\big(p_\nc(A_\c)^{-1} X+X p_\nc(A_\c)^{-1}\big).
\end{array}
\end{equation}
The last equality holds because for $\Theta(t)=\sum_{k=0}^N a_k t^k$, we have 
$$
\widetilde{\mathsf{R}}_{g_H,\Theta,A_\c}(X)=\sum_{k=1}^N a_k\sum_{i+j=k-1}A_\c^i X A_\c^j
$$
and $p_\nc(A_\c)^{-1}$ commutes with $A_\c$.
Let us replace $X$ in $\widetilde{\mathsf{R}}_{g_H,\Theta/p_\nc,A_\c}(X)$ with the generating element 
$$
X=u\wedge_{g_H}v=g_H(u,\cdot)\otimes v-g_H(v,\cdot)\otimes u
$$ 
of $\mathfrak{so}(g_H)$, where $u,v \in \mathbb V$ are arbitrary. For this choice of $X$ we have
\begin{equation}\label{eq:secondeq}
p_\nc(A_\c)^{-1} X+X p_\nc(A_\c)^{-1}=u\wedge_{g_H}(p_\nc(A_\c)^{-1}v)+(p_\nc(A_\c)^{-1}u)\wedge_{g_H}v.
\end{equation}
We obtain
\begin{equation}\label{eq:Step1}
\begin{array}{c}
\displaystyle \tilde{R}_{g_H,\Theta/p_\nc,A_\c}(x,y,u,v)=\frac{1}{2}g_H(x,\widetilde{\mathsf{R}}_{g_H,\Theta/p_\nc,A_\c}(u\wedge_{g_H}v)y)
\vspace{1mm}\\
\displaystyle \overset{\eqref{eq:firsteq},\eqref{eq:secondeq}}{=}
\frac{1}{2}\tilde{R}_{g_H,\Theta,A_\c}(x,y,u,p_\nc(A_\c)^{-1}v)
+\frac{1}{2}\tilde{R}_{g_H,\Theta,A_\c}(x,y,p_\nc(A_\c)^{-1}u,v).
\end{array}
\end{equation}
Suppose $\Theta(t)=\sum_{k=0}^N a_k t^k$. We can easily verify (using \eqref{eq:RABtilde} and \eqref{eq:polcurvopRiem0})
that the $(0,4)$-curvature tensor corresponding to the operator $\widetilde{\mathsf{R}}_{g_H,\Theta,A_\c}$ takes the form
\begin{equation}\label{eq:RThetaA04}
\tilde{R}_{g_H,\Theta,A_\c}=-\frac{1}{4}\sum_{k=1}^N a_k\sum_{r+s=k-1}g_H A_\c^r\owedge g_HA_\c^s.
\end{equation}
Thus, we compute (using $g_H=g_\c(p_\nc(A_\c)\cdot,\cdot)$)
\begin{equation}\label{eq:Step2}
\begin{array}{c}
\displaystyle \frac{1}{2}\tilde{R}_{g_H,\Theta,A_\c}(x,y,u,p_\nc(A_\c)^{-1}v)=
\vspace{1mm}\\
\displaystyle =-\frac{1}{4}\sum_{k=1}^N a_k\sum_{r+s=k-1}
[
g_\c (A_\c^r p_\nc(A_\c)x,u)g_\c (A_\c^s y,v)
-g_\c (A_\c^r x,v)g_\c (A_\c^s p_\nc(A_\c)y,u)
]
\end{array}
\end{equation}
and similarly
\begin{equation}\label{eq:Step3}
\begin{array}{c}
\displaystyle \frac{1}{2}\tilde{R}_{g_H,\Theta,A_\c}(x,y,p_\nc(A_\c)^{-1}u,v)=
\vspace{1mm}\\
\displaystyle =-\frac{1}{4}\sum_{k=1}^N a_k\sum_{r+s=k-1}
[
g_\c (A_\c^r x,u)g_\c (A_\c^s p_\nc(A_\c)y,v)
-g_\c (A_\c^r p_\nc(A_\c)x,v)g_\c (A_\c^s y,u)
].
\end{array}
\end{equation}
Inserting \eqref{eq:Step2} and \eqref{eq:Step3} into \eqref{eq:Step1} yields
\begin{equation}\label{eq:Step4}
\begin{array}{c}
\displaystyle \tilde{R}_{g_H,\Theta/p_\nc,A_\c}(x,y,u,v)=
\vspace{1mm}\\
\displaystyle =\tfrac{1}{2}\tilde R_{g_\c,\Theta,A_\c}(p_\nc(A_\c)x,y,u,v)+\tfrac{1}{2}\tilde R_{g_\c,\Theta,A_\c}(x,p_\nc(A_\c)y,u,v).
\end{array}
\end{equation}
Now $\tilde R_{g_\c,\Theta,A_\c}(x,p_\nc(A_\c)y,u,v)=g_\c(x,\tilde R_{g_\c,\Theta,A_\c}(u,v)p_\nc(A_\c)y)$ in $(1,3)$-tensor notation.
Since $[\tilde{R}_{g_\c,\Theta,A_\c}(u,v),p_\nc(A_\c)]=0$ (because of $[\tilde{R}_{g_\c,\Theta,A_\c}(u\wedge v),A_\c]=0$), we obtain 
$$
\tilde R_{g_\c,\Theta,A_\c}(x,p_\nc(A_\c)y,u,v)=g_\c(p_\nc(A_\c)x,\tilde R_{g_\c,\Theta,A_\c}(u,v)y)=\tilde R_{g_\c,\Theta,A_\c}(p_\nc(A_\c)x,y,u,v).
$$
Then, \eqref{eq:Step4} can be rewritten as
\begin{equation}\label{eq:thirdeq}
\tilde{R}_{g_H,\Theta/p_\nc,A_\c}(x,y,u,v)=\tilde R_{g_\c,\Theta,A_\c}(p_\nc(A_\c)x,y,u,v).
\end{equation}
So far the discussion for the Riemannian curvature tensors. We now apply the projection operator
${\mathsf{pr}}:\mathcal{R}(\mathbb V)\rightarrow \mathcal{K}(\mathbb V)$  from Lemma~\ref{lem:projection} to both sides of \eqref{eq:thirdeq}
and obtain 
\begin{equation}\label{eq:fortheq}
\begin{array}{c}
R_{g_H,\Theta/p_\nc,A_\c}(x,y,u,v)=2\,{\mathsf{pr}}(\tilde{R}_{g_H,\Theta/p_\nc,A_\c})(x,y,u,v)
\vspace{1mm}\\
=2\,{\mathsf{pr}}(\tilde{R}_{g_\c,\Theta,A_\c})(p_\nc(A_\c)x,y,u,v)=R_{g_\c,\Theta,A_\c}(p_\nc(A_\c)x,y,u,v)
\end{array}
\end{equation}
(we used that $p_\nc(A_\c)$ commutes with $J$).
Finally, we conclude
$$
R_\c(p_\nc(A_\c)x,y,u,v)=g_\c(p_\nc(A_\c)x,R_\c(u,v)y)=g_H(x,R_\c(u,v)y)=R_{g_H}(x,y,u,v)
$$
$$
\overset{\eqref{eq:horizBFcond}}{=}R_{g_H,\Theta/p_\nc,A_\c}(x,y,u,v)
\overset{\eqref{eq:fortheq}}{=}R_{g_\c,\Theta,A_\c}(p_\nc(A_\c)x,y,u,v)
$$
for all $x,y,u,v\in \mathbb V$. Thus, \eqref{eq:horizBFcond} is equivalent to
$R_\c=R_{g_\c,\Theta,A_\c}$ as we claimed.
\end{proof}

\subsection*{\bf Acknowledgements.} We are grateful to D.M.J. Calderbank, V.S. Matveev and L. Schwachh\"ofer for discussions and comments concerning this paper. 
The work of the first author was supported by the Russian Science Foundation (grant No. 17-11-01303).
The second author thanks Deutsche Forschungsgemeinschaft (Research training group   1523 --- Quantum and Gravitational Fields), 
Friedrich-Schiller-Universit\"at Jena and Leibniz Universit\"at Hannover for partial financial support.

%\nocite{*}
%\bibliographystyle{plain}

\end{document}